\documentclass[12pt]{amsart}
\usepackage{amsmath, amssymb, amscd, bbm, mathrsfs, url, pinlabel, hyperref, verbatim, xargs}
\usepackage[margin=1.25in,centering,letterpaper,dvips,marginparwidth=1in,marginparsep=0.15in]{geometry}
\usepackage{color,dcpic,latexsym,graphicx,epstopdf,comment}
\usepackage[all]{xy}
\usepackage{enumitem}
\usepackage{tikz}
\usetikzlibrary{calc}

\newtheorem {theorem}{Theorem}
\newtheorem {lemma}[theorem]{Lemma}
\newtheorem {proposition}[theorem]{Proposition}
\newtheorem {corollary}[theorem]{Corollary}

\newtheorem {definition}[theorem]{Definition}

\theoremstyle{remark}

\newtheorem {remark}[theorem]{Remark}
\newtheorem {example}[theorem]{Example}

\numberwithin{equation}{section}
\numberwithin{theorem}{section}
\theoremstyle{definition}

\newlist{pcases}{enumerate}{1}
\setlist[pcases]{
  label=\bf{Case~\arabic*:}\protect\thiscase.~,
  ref=\arabic*,
  align=left,
  labelsep=0pt,
  leftmargin=0pt,
  labelwidth=0pt,
  parsep=0pt
}
\newcommand{\case}[1][]{%
  \if\relax\detokenize{#1}\relax
    \def\thiscase{}%
  \else
    \def\thiscase{~#1}%
  \fi
  \item
}

\newcommand{\ZZ}{\mathbb{Z}}
\newcommand{\Z}{\mathbb{Z}}
\newcommand{\R}{\mathbb{R}}

\newcommand{\C}{\mathbb{C}}

\newcommand{\Q}{\mathbb{Q}}

\newcommand{\RP}{\mathbb{RP}}

\newcommand{\re}{\operatorname{Re}}
\newcommand{\coker}{\operatorname{coker}}

\DeclareMathOperator{\Hom}{Hom}
\DeclareMathOperator{\tr}{tr}

\DeclareMathOperator{\Aut}{Aut}

\newcommand{\Id}{\mathrm{Id}}
\newcommand{\dcover}{\Sigma_2}

\newcommand{\llangle}{\langle \langle}
\newcommand{\rrangle}{\rangle \rangle}
\newcommand{\norm}[1]{\| #1 \|}
\newcommand{\bgnorm}[1]{\{ #1 \}}

\tikzset{every picture/.style=thick}
\tikzset{baseline=-\the\dimexpr\fontdimen22\textfont2\relax}
\tikzset{link/.style = { white, double = black, line width = 1.75pt, double distance = 1.25pt, looseness=1.75 }}
\tikzset{blacklabel/.style={draw, fill=black!20, font=\tiny, circle, inner sep = 0.075cm}}

\title{A menagerie of $SU(2)$-cyclic 3-manifolds}
\date{}

\author{Steven Sivek}
\address{Department of Mathematics \\ Imperial College London}
\email{s.sivek@imperial.ac.uk}

\author{Raphael Zentner}
\address{Universit\"{a}t Regensburg}
\email{raphael.zentner@mathematik.uni-regensburg.de}

\begin{document}

\begin{abstract}
We classify $SU(2)$-cyclic and $SU(2)$-abelian 3-manifolds, for which every representation of the fundamental group into $SU(2)$ has cyclic or abelian image respectively, among geometric 3-manifolds which are not hyperbolic.  As an application, we give examples of hyperbolic 3-manifolds which do not admit degree-1 maps to any Seifert fibered manifold other than $S^3$ or a lens space.  We also produce infinitely many one-cusped hyperbolic manifolds with at least four $SU(2)$-cyclic Dehn fillings, one more than the number of cyclic fillings allowed by the cyclic surgery theorem.
\end{abstract}

\maketitle

\section{Introduction}

This paper is devoted to studying a class of 3-manifolds which are as simple as possible from the perspective of instanton gauge theory.

\begin{definition} \label{def:su2-cyclic}
A 3-manifold $Y$ is \emph{$SU(2)$-cyclic} if every representation $\rho: \pi_1(Y) \to SU(2)$ has cyclic image,  and it is \emph{$SU(2)$-abelian} if every $\rho$ has abelian image.
\end{definition}

Certainly $SU(2)$-cyclic 3-manifolds are $SU(2)$-abelian, and it is not hard to see that the two notions coincide for rational homology spheres.  Note that $T^3$ is $SU(2)$-abelian but not $SU(2)$-cyclic, however, because there are representations
\[ \pi_1(T^3) \cong \Z^3 \to SU(2) \]
which send the generators of the three $\Z$ summands to elements $\left(\begin{smallmatrix} e^{i\theta_j} & 0 \\ 0 & e^{-i\theta_j}\end{smallmatrix}\right)$ for some constants $\theta_1,\theta_2,\theta_3$ which are rationally independent in $\R/2\pi\Z$.

Questions about $SU(2)$-cyclic 3-manifolds go back at least to Kirby's problem list \cite{kirby-list}, in which Problem 3.105(A) asks whether any such homology spheres exist other than $S^3$.  Kronheimer and Mrowka's proof of the Property P conjecture \cite{km-p} and their closely related work \cite{km-su2} established that many surgeries on nontrivial knots in $S^3$ are not $SU(2)$-cyclic.  Lin \cite{lin} proved an inequality relating the slopes of two $SU(2)$-cyclic surgeries on a knot in $S^3$, Baldwin and the first author \cite{bs-stein} gave an obstruction to being $SU(2)$-cyclic in terms of Stein fillings, and the second author showed that the splicing of any two non-trivial knots in $S^3$ is not $SU(2)$-cyclic in \cite{zentner}. 

In this paper, we search systematically for examples of $SU(2)$-abelian $3$-manifolds, and in particular those with geometric structures.  Our first result is the following.

\begin{theorem} \label{thm:su2-abelian-sfs}
Let $Y$ be a closed, orientable, Seifert fibered 3-manifold.  Then $Y$ is $SU(2)$-abelian if and only if one of the following holds:
\begin{itemize}
\item $Y$ is $S^3$, a lens space, $S^1\times S^2$, or $T^3$, hence $\pi_1(Y)$ is abelian.
\item $Y$ is $\RP^3\#\RP^3$.
\item $Y$ has base orbifold $S^2(2,4,4)$.
\item $Y$ has base orbifold $S^2(3,3,3)$, and $|H_1(Y;\Z)|$ is either even or infinite.
\item $Y$ is a circle bundle over $T^2$ with nonzero, even Euler number.
\end{itemize}
\end{theorem}

\begin{remark}
The only manifolds $Y$ in Theorem~\ref{thm:su2-abelian-sfs} for which $H_1(Y;\Z)$ is cyclic are $S^3$, lens spaces, and $S^1\times S^2$.  In particular, an $SU(2)$-cyclic surgery on a nontrivial knot in $S^3$ has cyclic $H_1$, so if it is Seifert fibered then it must be a lens space.
\end{remark}

\begin{remark}
In the case where the base is $S^2(3,3,3)$ and $|H_1(Y;\Z)|$ is even, we further show (Proposition~\ref{prop:non-cyclically-finite}) that $\pi_1(Y)$ is not \emph{cyclically finite}: $Y$ has a normal cover of degree 3 which is a circle bundle over $T^2$, hence not a rational homology sphere.
\end{remark}

Seifert fibered spaces comprise six of the eight geometric structures on 3-manifolds, and a seventh is not so hard to understand either.

\begin{theorem}[Theorem~\ref{thm:count-sol-bundles}]
There are exactly three closed, orientable $\mathrm{Sol}$-manifolds which are $SU(2)$-abelian, and all of them are $T^2$-bundles over $S^1$.
\end{theorem}
In Theorem~\ref{thm:count-sol-bundles} we also explicitly determine the monodromies of these bundles up to conjugacy.

We know very little about the remaining case, namely hyperbolic manifolds, but we can prove that some closed hyperbolic 3-manifolds are $SU(2)$-cyclic by the following, which was claimed without proof in \cite{zentner-simple} and builds on work of Cornwell \cite{cornwell}.

\begin{theorem}[Theorem~\ref{thm: bridge number 3}]
If $K \subset S^3$ is an $SU(2)$-simple knot (see Definition~\ref{def:su2-simple}) of bridge number at most 3, then its branched double cover $\dcover(K)$ is $SU(2)$-cyclic. In particular, this proves that the hyperbolic manifolds $\dcover(8_{18})$ and $\dcover(10_{109})$ are both $SU(2)$-cyclic. 
\end{theorem}

Together with our classification result in Theorem~\ref{thm:su2-abelian-sfs}, this gives an obstruction to the existence of degree-1 maps (also called 1-dominations) from some hyperbolic 3-manifolds to any Seifert fibered 3-manifold other than $S^3$ or a lens space.

\begin{theorem}[Theorem~\ref{thm:obstruction-1-domination}]
Any $SU(2)$-cyclic hyperbolic rational homology 3-sphere $Y$, where $H_1(Y;\Z)$ has odd order, admits no map of degree 1 to any Seifert fibered 3-manifold with 3 or more singular fibers.
Examples of such hyperbolic manifolds include the branched double covers of the knots $8_{18}$ and $10_{109}$. 
\end{theorem}

To put this last result into context, we note that no Seifert fibered 3-manifold can 1-dominate a hyperbolic manifold.  This follows from basic properties of the simplicial volume $\norm{M}$ of a 3-manifold $M$ \cite{Gromov-volume}.  Indeed, if $Y$ 1-dominates $Y'$ then we have an inequality $\norm{Y} \geq \norm{Y'}$.  But if $Y$ is Seifert fibered then it has zero simplicial volume, whereas if $Y'$ is hyperbolic then $\norm{Y'}$ is proportional to $\operatorname{vol}(Y')$ and hence positive.  In fact, simplicial volume also is additive under decompositions along tori, so no graph manifold can 1-dominate a hyperbolic 3-manifold either.

In the opposite direction, every Seifert fibered 3-manifold is 1-dominated by some hyperbolic manifold.  This is a corollary of recent work of Liu and Sun \cite{Liu-Sun}, who showed that any hyperbolic 3-manifold virtually 1-dominates any other 3-manifold.

Similarly, Brooks and Goldman \cite{Brooks-Goldman} defined a notion of Seifert volume $\bgnorm{M}$, which also satisfies $\bgnorm{Y} \geq \bgnorm{Y'}$ if $Y$ 1-dominates $Y'$.  They showed that $\bgnorm{Y}=0$ for many hyperbolic $Y$, and that $\bgnorm{Y'} > 0$ for $Y'$ with $\widetilde{SL(2,\R)}$-geometry (such as any Brieskorn sphere other than than the Poincar\'e homology sphere).  As a consequence, there are hyperbolic manifolds which do not 1-dominate any Seifert fibered manifold having either $\widetilde{SL(2,\R)}$-geometry or at least four singular fibers, the latter being obstructed by the presence of incompressible tori.  Our examples are stronger in the sense that they give an obstruction to 1-domination of \emph{all} Seifert fibered 3-manifolds with three or more singular fibers. 

\begin{remark}
Interestingly, the manifolds $\dcover(8_{18})$ and $\dcover(10_{109})$ also arise as Dehn fillings of the hyperbolic census manifolds $m036$ and $m100$ respectively, each of which has another $SU(2)$-cyclic filling at distance three from these.  This is the maximum known distance between two $SU(2)$-cyclic Dehn fillings of a one-cusped hyperbolic manifold.  See Examples~\ref{ex:m036} and \ref{ex:m100} for details.
\end{remark}

In previous work \cite{sivek-zentner}, we investigated the question of whether knots in $S^3$ other than torus knots can have infinitely many $SU(2)$-cyclic surgeries.  We still do not know the answer, or indeed whether any such knot can have more than three nontrivial $SU(2)$-cyclic surgeries, though in Example~\ref{ex:22n0-surgeries} we present two hyperbolic knots in $S^3$ with three $SU(2)$-cyclic surgeries each: these are the pretzel knot $P(-2,3,7)$ and a twisted torus knot labeled $k4_4$ in the Callahan--Dean--Weeks census \cite{callahan-dean-weeks}.  More generally, we can prove the following.
\begin{theorem}[Theorem~\ref{thm:four-fillings} and Theorem~\ref{thm:M_g-hyperbolic}]
There are infinitely many compact, oriented, hyperbolic 3-manifolds with torus boundary that have at least four $SU(2)$-cyclic Dehn fillings.
\end{theorem}
The cyclic surgery theorem \cite{cgls} says that each of these can only have three cyclic Dehn fillings, and all of the manifolds in Example~\ref{ex:22n0-surgeries} and in Theorem~\ref{thm:four-fillings} achieve this upper bound; the fourth $SU(2)$-cyclic filling is a graph manifold.  Some further examples with four $SU(2)$-cyclic fillings, only two of which are cyclic, are presented in Examples~\ref{ex:m036} and \ref{ex:four-sporadic}.

\subsection*{Organization}

The proof of Theorem~\ref{thm:su2-abelian-sfs} consists of several parts.  Section~\ref{sec:sfs-over-s2} is devoted to the proof of Theorem~\ref{thm:sfs-over-s2}, which handles the cases where the base orbifold is $S^2$ with some number of singular fibers; this includes the cases where $\pi_1(Y)$ is cyclic and also those with base orbifold $S^2(2,4,4)$ or $S^2(3,3,3)$.  The main result of Section~\ref{sec:sfs-over-other-bases}, Theorem~\ref{thm:sfs-over-other-bases}, addresses other base orbifolds, which have to be either $\RP^2$ or $T^2$ and which yield $\RP^3\#\RP^3$ and the circle bundles over $T^2$ (including $T^3$) respectively.

Sections~\ref{sec:sol} and \ref{sec:hyperbolic} address the cases of $\mathrm{Sol}$ and hyperbolic geometry, respectively.  In the former we obtain a complete classification; in the latter we prove Theorem~\ref{thm: bridge number 3}, giving a criterion for proving that the branched double cover of a knot is $SU(2)$-cyclic, but we can only apply it to sporadic examples.  We apply these examples to questions about 1-domination in Theorem~\ref{thm:obstruction-1-domination}.  Then in Section~\ref{sec:cyclic-covers} we show that $SU(2)$-cyclicity is not preserved under finite cyclic covers.  Finally, in Section~\ref{sec:four-fillings} we present examples of 3-manifolds with torus boundary and many $SU(2)$-cyclic fillings.

Throughout this paper, we will identify $SU(2)$ with the group of unit quaternions via the map $\left(\begin{smallmatrix} \alpha & \beta \\ -\bar\beta & \bar\alpha \end{smallmatrix}\right) \mapsto \alpha + \beta j$, and $\R^3$ with the space of purely imaginary quaternions with basis $i,j,k$.  Any element of $SU(2)$ can be written in the form $\cos(\theta) + v\sin(\theta)$, where $v\in \R^3$ is a unit vector, and the product of any $v,w \in \R^3$ is $vw = -\langle v,w\rangle + (v\times w)$, with $\re(vw) = -\langle v,w\rangle$.

\subsection*{Acknowledgments}

We thank Chris Cornwell for some helpful discussions and for sharing his Mathematica code from \cite{cornwell} with us. We also thank Yi Liu and Hongbin Sun for helpful discussions, and the referees for a careful reading and for suggesting Lemma~\ref{lem:triangle}. The second author is also grateful for support by the SFB ‘Higher invariants’ (funded by the Deutsche Forschungsgemeinschaft (DFG)) at the University of Regensburg, and for support by a Heisenberg fellowship of the DFG.

\section{Seifert fibered manifolds over $S^2$} \label{sec:sfs-over-s2}

Our goal in this section is to classify the $SU(2)$-cyclic Seifert fibered manifolds with base orbifold $S^2$ and any number of singular fibers.  Following \cite{jankins-neumann}, we write
\[ Y = S^2((\alpha_1,\beta_1),(\alpha_2,\beta_2),\dots,(\alpha_n,\beta_n)) \]
with $\gcd(\alpha_i,\beta_i)=1$ and $\alpha_i \geq 1$ for all $i$.  Then by \cite[Theorem~6.1]{jankins-neumann}, we have
\begin{equation} \label{eq:pi1-sfs}
\pi_1(Y) = \langle c_1,c_2,\dots,c_n,h \mid [h,c_i]=c_i^{\alpha_i}h^{\beta_i}=1\,\,\forall i, c_1c_2\dots c_n=1 \rangle.
\end{equation}
The order of the $(\alpha_i,\beta_i)$ does not matter, so throughout this section, we will assume that $\alpha_1 \leq \alpha_2 \leq \dots \leq \alpha_n$.  If $0 \leq n \leq 2$ then $Y$ is a lens space, $S^3$, or $S^1\times S^2$, which we already know to be cyclic and hence $SU(2)$-cyclic, so we can always take $n \geq 3$.  Moreover, by \cite[Theorem~1.5]{jankins-neumann}, we can also replace
\[ (1,k), (\alpha,\beta) \leadsto (1,0), (\alpha,\beta+k\alpha) \leadsto (\alpha,\beta+k\alpha) \]
without changing $Y$, so throughout this section we will also assume that $\alpha_1 \geq 2$.  (This substitution reduces $n$ by $1$, but if it results in a fibration with $n < 3$ then $\pi_1(Y)$ is again cyclic and so we can continue to assume $n \geq 3$.)

The main result of this section is the following.
\begin{theorem} \label{thm:sfs-over-s2}
Let $Y$ be a Seifert fiber space with base $S^2$ and any number of singular fibers.  Then $Y$ is $SU(2)$-abelian if and only if one of the following is true:
\begin{itemize}
\item $Y$ is a lens space, $S^3$, or $S^1\times S^2$;
\item $Y$ has base orbifold $S^2(2,4,4)$;
\item $Y$ has base orbifold $S^2(3,3,3)$ and $|H_1(Y;\Z)|$ is even or infinite.
\end{itemize}
\end{theorem}

\begin{proof}
Let $n$ be the number of singular fibers.  In the case where $n \leq 2$ there is nothing to show.  The case $n=3$ is Theorem~\ref{thm:small-seifert-fibered}, while Theorem~\ref{thm:large-seifert-fibered} asserts that there are no further examples when $n \geq 4$.
\end{proof}

As a preliminary step, we note that the fundamental group $\pi_1(Y)$ surjects onto the orbifold fundamental group
\begin{align} \label{eq:pi1-triangle}
\begin{split}
\Delta(\alpha_1,\dots,\alpha_n) &= \pi_1(Y)/\llangle h\rrangle \\
&= \langle c_1,\dots,c_n \mid c_i^{\alpha_i}=1\ \forall i, c_1\dots c_n=1 \rangle.
\end{split}
\end{align}
In many cases, we will build non-abelian representations of $\pi_1(Y)$ by composing this surjection with a non-abelian representation of $\Delta(\alpha_1,\dots,\alpha_n)$.  The following lemma will be useful in the remaining cases.

\begin{lemma} \label{lem:triangle-h-image}
If there are no non-abelian representations
\[ \Delta(\alpha_1,\dots,\alpha_n) \to SU(2), \]
then every non-abelian representation $\rho: \pi_1(Y) \to SU(2)$ satisfies $\rho(h)=-1$.
\end{lemma}

\begin{proof}
Since $h$ is central, it must commute with the image of $\rho$, hence $\rho(h)$ lies in the center $\{\pm 1\}$ of $SU(2)$.  If $\rho(h)=1$ then $\rho$ induces a representation of $\pi_1(Y)/\llangle h\rrangle$ with the same non-abelian image, and by assumption this does not exist, so we must have $\rho(h)=-1$ instead.
\end{proof}

\subsection{Representations of the orbifold fundamental group}

In this subsection, we fix a base orbifold $B = S^2(\alpha_1,\dots,\alpha_n)$, with $n \geq 3$ and $2 \leq \alpha_1 \leq \dots \leq \alpha_n$, and determine when its fundamental group has a non-abelian $SU(2)$-representation.  We begin with what will turn out to be nearly all cases in which no such representation exists.

\begin{lemma} \label{lem:triangle-start-2}
If $(\alpha_1,\dots,\alpha_n) = (2,\dots,2,p,q)$ for some integers $q \geq p \geq 2$, then every representation
\[ \Delta(\alpha_1,\dots,\alpha_n) \to SU(2) \]
has abelian image.
\end{lemma}

\begin{proof}
The generators $c_1,\dots,c_{n-2}$ in the presentation~\eqref{eq:pi1-triangle} are sent to elements of square $1$, so we have $\rho(c_i)=\pm1$ for all $i \leq n-2$, and then the relation $\rho(c_1\dots c_n)=1$ becomes $\rho(c_{n-1})\rho(c_n) = \pm 1$.  But then $\rho(c_n) = \pm \rho(c_{n-1})^{-1}$, so $\rho(c_{n-1})$ and $\rho(c_n)$ commute with each other and thus the image of $\rho$ must be abelian as claimed.
\end{proof}

We will now understand the case $n=3$ and then proceed to larger $n$ by induction.  We are grateful to one of the referees for suggesting the following lemma as a way to simplify the arguments throughout the rest of the section.

\begin{lemma} \label{lem:triangle}
Consider the free group $F$ on two generators, with presentation
\[ F = \langle x_1,x_2,x_3 \mid x_1x_2x_3 = 1\rangle. \]
Given three real numbers $\theta_1,\theta_2,\theta_3 \in [0,\pi]$, there are imaginary unit quaternions $v_1,v_2,v_3$ and a representation $\rho: F \to SU(2)$ satisfying
\[ \rho(x_i) = e^{v_i\theta_i} = \cos(\theta_i) + \sin(\theta_i) v_i, \qquad 1 \leq i \leq 3 \]
if and only if all of the inequalities
\begin{align} \label{eq:theta-triangle}
\theta_1+\theta_2 &\geq \theta_3, &
\theta_2+\theta_3 &\geq \theta_1, &
\theta_3+\theta_1 &\geq \theta_2
\end{align}
and 
\begin{equation} \label{eq:theta-sum}
\theta_1+\theta_2+\theta_3 \leq 2\pi
\end{equation}
are satisfied.  If $\rho$ exists then 
it has non-abelian image if and only if all four of these inequalities are strict.
\end{lemma}

\begin{proof}
Given arbitrary unit quaternions $v_1$ and $v_2$ we let $z=e^{v_1\theta_1}e^{v_2\theta_2}$, and we compute that
\begin{equation*} 
\re(z) = \cos(\theta_1)\cos(\theta_2) - \sin(\theta_1)\sin(\theta_2)\langle v_1,v_2\rangle.
\end{equation*}
Since $\sin(\theta_1)$ and $\sin(\theta_2)$ are both nonnegative and $\big| \langle v_1,v_2\rangle \big| \leq \pm 1$ with equality iff $v_2=\pm v_1$, it follows that
\[ \cos(\theta_1+\theta_2) \leq \re(z) \leq \cos(\theta_1-\theta_2). \]
By varying $\langle v_1,v_2\rangle$ we can take $\re(z)$ to be any value in this range, and it achieves either of the endpoints if and only if
\begin{multline*}
\sin(\theta_1)\sin(\theta_2)\langle v_1,v_2\rangle = \pm \sin(\theta_1)\sin(\theta_2) \\
\begin{array}{l}
\Longleftrightarrow\text{ any of }\sin(\theta_1)=0,\ \sin(\theta_2)=0,\text{ or }\langle v_1,v_2\rangle=\pm1\text { are true} \\[0.5ex]
\Longleftrightarrow\ e^{v_1\theta_1} \text{ commutes with } e^{v_2\theta_2}.
\end{array}
\end{multline*}
Now the existence of $\rho$ is equivalent to finding $v_1,v_2,v_3$ such that $z = e^{(-v_3)\theta_3}$, and given $v_1$ and $v_2$ we can find $v_3$ if and only if $\re(z) = \cos(\theta_3)$, so $\rho$ exists if and only if
\begin{equation} \label{eq:re-ijk}
\cos(\theta_1+\theta_2) \leq \cos(\theta_3) \leq \cos(\theta_1-\theta_2),
\end{equation}
and then $\rho$ is abelian if and only if at least one of these inequalities is an equality.

We rewrite \eqref{eq:re-ijk} as
\[ \cos(\alpha) \leq \cos(\theta_3) \leq \cos(\beta), \]
where $\alpha,\beta \in [0,\pi]$ are defined by
\[ \alpha = \begin{cases} \theta_1+\theta_2, & \theta_1+\theta_2 \leq \pi \\ 2\pi-\theta_1-\theta_2, & \theta_1+\theta_2 > \pi \end{cases}
\qquad\text{and}\qquad
\beta = \begin{cases} \theta_1-\theta_2, & \theta_1 \geq \theta_2 \\ \theta_2-\theta_1, & \theta_1<\theta_2. \end{cases} \]
Then \eqref{eq:re-ijk} is satisfied if and only if $\beta \leq \theta_3 \leq \alpha$, since $\cos(\theta)$ is strictly monotone decreasing on $[0,\pi]$.  Since $\beta = \max(\theta_1-\theta_2,\theta_2-\theta_1)$, we have
\begin{align*}
\beta \leq \theta_3 \ &\Longleftrightarrow\ \theta_1-\theta_2 \leq \theta_3 \text{ and } \theta_2-\theta_1 \leq \theta_3 \\
&\Longleftrightarrow\ \theta_2+\theta_3 \geq \theta_1 \text{ and } \theta_3+\theta_1 \geq \theta_2,
\end{align*}
and likewise for $\alpha = \min(\theta_1+\theta_2, 2\pi-\theta_1-\theta_2)$ we have
\begin{align*}
\theta_3 \leq \alpha \ &\Longleftrightarrow\ \theta_3 \leq \theta_1+\theta_2 \text{ and } \theta_3 \leq 2\pi-\theta_1-\theta_2 \\
&\Longleftrightarrow\ \theta_1+\theta_2 \geq \theta_3 \text{ and } \theta_1+\theta_2+\theta_3 \leq 2\pi.
\end{align*}
So the desired $v_1$, $v_2$, and $v_3$ exist if and only if the inequalities \eqref{eq:theta-triangle} and \eqref{eq:theta-sum} are satisfied.  Moreover, we have either $\alpha=\theta_3$ or $\beta=\theta_3$, which is equivalent to $\rho$ having abelian image, if and only if one of the four inequalities is in fact an equality.
\end{proof}

\begin{proposition} \label{prop:triangle-3}
Suppose that $n=3$ and $\alpha_1 \geq 3$. Then there is a representation $\rho: \Delta(\alpha_1,\alpha_2,\alpha_3) \to SU(2)$ with non-abelian image if and only if $(\alpha_1,\alpha_2,\alpha_3) \neq (3,3,3)$.
\end{proposition}

\begin{proof}
We recall that $\alpha_1 \leq \alpha_2 \leq \alpha_3$ by convention, so the condition $(\alpha_1,\alpha_2,\alpha_3) \neq (3,3,3)$ is equivalent to $\alpha_3 \geq 4$.  The presentation~\eqref{eq:pi1-triangle} says that
\[ \Delta(\alpha_1,\alpha_2,\alpha_3) = \langle c_1,c_2,c_3 \mid c_1^{\alpha_1} = c_2^{\alpha_2} = c_3^{\alpha_3} = 1,\ c_1c_2c_3 = 1 \rangle, \]
so we attempt to set $\rho(c_i) = \cos(\theta_i) + v_i \sin(\theta_i)$ for imaginary unit quaternions $v_i$ and angles $\theta_i \in [0,\pi]$ such that $\alpha_i\theta_i\in 2\pi\Z$, and then we will apply Lemma~\ref{lem:triangle} to determine whether we have succeeded.

Assuming first that $\alpha_3 \geq 4$, we take $\theta_1 = \frac{2\pi}{\alpha_1}$ and $\theta_2 = \frac{2\pi}{\alpha_2}$, and we set
\[ \theta_3 = \min\left\{ \frac{2\pi m}{\alpha_3} \,\middle\vert\, m\in\Z,\ \frac{2\pi m}{\alpha_3} > \frac{2\pi}{\alpha_1}-\frac{2\pi}{\alpha_2} \right\}. \]
Then $\theta_1 \geq \theta_2$ and $\theta_3 > \theta_1 - \theta_2 \geq 0$ by definition, so
\[ \theta_3+\theta_1 > \theta_1 \geq \theta_2 \quad\text{and}\quad \theta_2+\theta_3 > \theta_1. \]
We also have an inequality $\theta_3 \leq \theta_1 - \theta_2 + \frac{2\pi}{\alpha_3}$, from which we deduce
\[ \theta_2 = \frac{2\pi}{\alpha_2} \geq \frac{2\pi}{\alpha_3} \ \Longrightarrow\ \theta_3 \leq \theta_1 - \left(\theta_2 - \frac{2\pi}{\alpha_3}\right) \leq \theta_1 < \theta_1 + \theta_2 \]
(note that $\theta_3 \leq \theta_1$ guarantees $\theta_3 \in [0,\pi]$) and also
\[ \theta_1+\theta_2+\theta_3 \leq 2\theta_1 + \frac{2\pi}{\alpha_3} \leq 2\pi\left(\frac{2}{\alpha_1} + \frac{1}{\alpha_3}\right) < 2\pi, \]
the last inequality using the assumptions $\alpha_1 \geq 3$ and $\alpha_3 \geq 4$.  We have verified each of the strict inequalities of \eqref{eq:theta-triangle} and \eqref{eq:theta-sum}, so we conclude by Lemma~\ref{lem:triangle} that $\rho$ exists and has non-abelian image.

Now we wish to show that there are no non-abelian $\rho$ when $(\alpha_1,\alpha_2,\alpha_3) \neq (3,3,3)$.  If we have a representation $\rho$ with $\rho(c_i) = e^{v_i\theta_i}$ for some angles $\theta_i \in [0,\pi]$, and $\rho$ has non-abelian image, then Lemma~\ref{lem:triangle} says that $\theta_1+\theta_2+\theta_3 < 2\pi$.  On the other hand, the conditions $\rho(c_i^3) = 1$ translate to $\theta_i \in \frac{2\pi}{3}\Z$ for each $i$, and if $\rho$ has non-abelian image then all of the $\theta_i$ are positive, so we must have $\theta_i \geq \frac{2\pi}{3}$ for each $i$.  But then
\[ \theta_1 + \theta_2 + \theta_3 \geq 2\pi, \]
which is a contradiction, so there are no non-abelian representations after all.
\end{proof}

\begin{theorem} \label{thm:triangle-reps}
Suppose that $n \geq 3$ and that $2 \leq \alpha_1 \leq \alpha_2 \leq \dots \leq \alpha_n$.  Then there is a non-abelian representation
\[ \Delta(\alpha_1,\dots,\alpha_n) \to SU(2) \]
if and only if $(\alpha_1,\dots,\alpha_n)$ does not have the form $(2,2,\dots,2,p,q)$ or $(3,3,3)$.
\end{theorem}

\begin{proof}
The case $n=3$ follows from Lemma~\ref{lem:triangle-start-2} and Proposition~\ref{prop:triangle-3}, so we assume from now on that $n \geq 4$.  By Lemma~\ref{lem:triangle-start-2} we can also assume that $\alpha_{n-2} \geq 3$.  If $(\alpha_{n-2},\alpha_{n-1},\alpha_n) \neq (3,3,3)$, then we simply compose the surjection
\begin{align*}
\Delta(\alpha_1,\dots,\alpha_n) &\to \Delta(\alpha_{n-2},\alpha_{n-1},\alpha_n) \\
c_1,c_2,\dots,c_{n-3} &\mapsto 1
\end{align*}
with a non-abelian representation of the latter.  This completes the proof except in the case $(\alpha_{n-2},\alpha_{n-1},\alpha_n)=(3,3,3)$.

Now supposing that $(\alpha_{n-2},\alpha_{n-1},\alpha_n)=(3,3,3)$, we have a surjection
\[ \Delta(\alpha_1,\dots,\alpha_n) \to \Delta(\alpha_{n-3},\alpha_{n-2},\alpha_{n-1},\alpha_n) \]
and $\alpha_{n-3}$ is either $2$ or $3$, so it suffices to find non-abelian $SU(2)$ representations of both $\Delta(2,3,3,3)$ and $\Delta(3,3,3,3)$.  For $\Delta(3,3,3,3)$, we take
\begin{align*}
\rho: \langle c_1,\dots,c_4 \mid c_i^3=1\ \forall i, c_1c_2c_3c_4=1\rangle &\to SU(2) \\
(c_1,c_2,c_3,c_4) &\mapsto (e^{2\pi i/3}, e^{2\pi j/3}, e^{-2\pi j/3}, e^{-2\pi i/3})
\end{align*}
and verify that $\rho(c_i)^3=1$ for all $i$ and $\rho(c_1c_2c_3c_4)=1$.  For $\Delta(2,3,3,3)$, we take
\[ c_1 \mapsto -1, \qquad c_i \mapsto \cos\left(\tfrac{2\pi}{3}\right) + \sin\left(\tfrac{2\pi}{3}\right) v_i \quad (i=2,3,4) \]
where the $v_i$ are purely imaginary unit quaternions
\begin{align*}
v_2 &= i, & v_3 &= \frac{-i+2j+2k}{3}, & v_4 &= \frac{-i + (-1-\sqrt{3})j + (-1+\sqrt{3})k}{3},
\end{align*}
and it is again straightforward to check that $\rho(c_1)^2 = 1$ and $\rho(c_i)^3=1$ for $i=2,3,4$, and (slightly more tediously) that $\rho(c_1c_2c_3c_4)=1$.
\end{proof}

\subsection{Three singular fibers} \label{sec:small-sfs}

In this subsection, we classify $SU(2)$-cyclic Seifert fiber spaces with base $S^2$ and exactly three singular fibers.  Again, we write
\begin{equation} \label{eq:Y-three-singular-fibers}
Y = S^2((\alpha_1,\beta_1),(\alpha_2,\beta_2),(\alpha_3,\beta_3))
\end{equation}
with $2 \leq \alpha_1 \leq \alpha_2 \leq \alpha_3$ and $\gcd(\alpha_i,\beta_i)=1$ for all $i$, and by \eqref{eq:pi1-sfs} we have a presentation
\[ \pi_1(Y) = \langle c_1,c_2,c_3,h \mid [h,c_i]=c_i^{\alpha_i}h^{\beta_i}=1\,\,\forall i, c_1c_2c_3=1 \rangle. \]

\begin{theorem} \label{thm:small-seifert-fibered}
Let $Y$ be a Seifert fiber space of the form \eqref{eq:Y-three-singular-fibers}.  Then $Y$ is $SU(2)$-abelian if and only if either
\begin{itemize}
\item $(\alpha_1,\alpha_2,\alpha_3) = (2,4,4)$, or
\item $(\alpha_1,\alpha_2,\alpha_3) = (3,3,3)$ and $|H_1(Y;\Z)|$ is either even or infinite.
\end{itemize}
\end{theorem}

\begin{proof}
If $\alpha_1 \geq 3$ but $(\alpha_1,\alpha_2,\alpha_3) \neq (3,3,3)$, then Theorem~\ref{thm:triangle-reps} gives us a representation
\[ \pi_1(Y) \to \pi_1(Y)/\llangle h\rrangle \to SU(2) \]
with non-abelian image.  If instead $\alpha_1 = 2$, then in Proposition~\ref{prop:some-alpha-is-2} we construct a non-abelian $\rho: \pi_1(Y) \to SU(2)$ in all cases except for $(\alpha_1,\alpha_2,\alpha_3)=(2,4,4)$, where we prove that no such $\rho$ exists.  Finally, in Proposition~\ref{prop:333-cyclic} we determine exactly which $Y$ admit non-abelian $SU(2)$ representations in the case $(\alpha_1,\alpha_2,\alpha_3)=(3,3,3)$.
\end{proof}

\begin{proposition} \label{prop:some-alpha-is-2}
If $\alpha_1 = 2$, then there is a representation $\rho: \pi_1(Y) \to SU(2)$ with non-abelian image if and only if $(\alpha_1,\alpha_2,\alpha_3) \neq (2,4,4)$.
\end{proposition}

\begin{proof}
Combining Theorem~\ref{thm:triangle-reps} with Lemma~\ref{lem:triangle-h-image} tells us that $\rho(h)=-1$, and moreover that $\rho(h)^{\beta_1} = -1$ since $\beta_1$ is coprime to $\alpha_1=2$.  A non-abelian representation $\rho: \pi_1(Y) \to SU(2)$ is thus equivalent to a non-abelian representation
\[ \rho': \langle c_1,c_2,c_3 \mid c_1c_2c_3 = 1 \rangle \]
with $\rho'(c_1)^2 = -1$, $\rho'(c_2)^{\alpha_2} = (-1)^{\beta_2}$, and $\rho'(c_3)^{\alpha_3} = (-1)^{\beta_3}$.  In particular, if we write $\rho'(c_i)=e^{v_i\theta_i}$ for some imaginary unit quaternions $v_i$ and angles $\theta_i \in [0,\pi]$, then these conditions amount to
\begin{align*}
\theta_1 &= \frac{\pi}{2}, &
\theta_2 &= \frac{(\beta_2 + 2m)\pi}{\alpha_2}, &
\theta_3 &= \frac{(\beta_3 + 2n)\pi}{\alpha_3}
\end{align*}
for some integers $m$ and $n$, and the $\theta_i$ must strictly satisfy the inequalities \eqref{eq:theta-triangle} and \eqref{eq:theta-sum} of Lemma~\ref{lem:triangle}.  We will accomplish this where possible by taking the $\theta_j$ as close to $\frac{\pi}{2}$ as possible.


For each of $j=2$ and $j=3$, we take
\[ \theta_j = \frac{(\lceil \alpha_j/2 \rceil -1)\pi}{\alpha_j} \text{ or } \frac{\lceil \alpha_j/2 \rceil \pi}{\alpha_j}, \]
depending on whether or not $\lceil \alpha_j/2 \rceil - 1 \equiv \beta_j\pmod{2}$.  We note that if $\alpha_j=2$, then $\beta_j$ must be odd and so $\theta_j=\frac{\pi}{2}$.  Otherwise, if $\alpha_j \geq 3$ then the larger of these values gives us
\[ \theta_j \leq \frac{((\alpha_j+1)/2)\pi}{\alpha_j} = \left(\frac{1}{2} + \frac{1}{2\alpha_j}\right)\pi \leq \frac{2\pi}{3}, \]
and the smaller value gives us a lower bound
\[ \theta_j \geq \begin{cases} \left(\frac{1}{2} - \frac{1}{2\alpha_j}\right) \pi, & \alpha_j\text{ odd} \\ \left(\frac{1}{2}-\frac{1}{\alpha_j}\right) \pi, & \alpha_j\text{ even}, \end{cases} \]
so combined we have $\frac{\pi}{4} \leq \theta_j \leq \frac{2\pi}{3}$ for any $\alpha_j \geq 2$, with equality on the left only if $\alpha_j=4$.  It follows that
\[ \theta_1+\theta_2+\theta_3 \leq \frac{\pi}{2}+\frac{2\pi}{3}+\frac{2\pi}{3} < 2\pi, \]
so the inequality \eqref{eq:theta-sum} holds strictly.  We also verify \eqref{eq:theta-triangle} by checking that
\begin{align*}
\theta_1+\theta_2 \geq \frac{\pi}{2} + \frac{\pi}{4} &= \frac{3\pi}{4} > \theta_3, \\
\theta_3+\theta_1 \geq \frac{\pi}{4} + \frac{\pi}{2} &= \frac{3\pi}{4} > \theta_2, \\
\theta_2+\theta_3 \geq \frac{\pi}{4} + \frac{\pi}{4} &= \frac{\pi}{2} = \theta_1,
\end{align*}
and all three inequalities are strict, except possibly the last one in the case $\theta_2=\theta_3=\frac{\pi}{4}$.  Thus the desired $\rho'$ exists and is non-abelian by Lemma~\ref{lem:triangle}, and hence likewise for $\rho$, except possibly when $\alpha_2=\alpha_3=4$.

Finally, in the case $(\alpha_1,\alpha_2,\alpha_3)=(2,4,4)$, we note that both $\beta_2$ and $\beta_3$ must be odd and so we are forced to take
\[ \theta_1 = \frac{\pi}{2}, \qquad \theta_2,\theta_3 \in \left\{\frac{\pi}{4},\frac{3\pi}{4}\right\}. \]
No matter how we choose $\theta_2$ and $\theta_3$, one of the inequalities in \eqref{eq:theta-triangle} and \eqref{eq:theta-sum} will in fact be an equality, so in this case Lemma~\ref{lem:triangle} says that any $\rho'$ must have abelian image, and hence so must any $\rho$.
\end{proof}

The sole remaining case with $n=3$ is $(\alpha_1,\alpha_2,\alpha_3) = (3,3,3)$, where the characterization of $SU(2)$-abelian $Y$ will make use of the following lemma.

\begin{lemma} \label{lem:homology-sfs}
The first homology of $Y$ has order
\[ |H_1(Y;\Z)| = |\alpha_1\alpha_2\beta_3 + \alpha_1\beta_2\alpha_3 + \beta_1\alpha_2\alpha_3|, \]
where we say that $|H_1(Y;\Z)|=0$ if and only if $H_1(Y;\Z)$ is infinite.
\end{lemma}

\begin{proof}
We abelianize the presentation \eqref{eq:pi1-sfs} of $\pi_1(Y)$ to compute that
\[ H_1(Y;\Z) = \coker \begin{pmatrix} \alpha_1 & 0 & 0 & \beta_1 \\ 0 & \alpha_2 & 0 & \beta_2 \\ 0 & 0 & \alpha_3 & \beta_3 \\ 1 & 1 & 1 & 0 \end{pmatrix}, \]
and the order of $H_1(Y;\Z)$ is given by the absolute value of the determinant $-\alpha_1\alpha_2\beta_3 - \alpha_1\beta_2\alpha_3 - \beta_1\alpha_2\alpha_3$ of this matrix.
\end{proof}

\begin{proposition} \label{prop:333-cyclic}
If $(\alpha_1,\alpha_2,\alpha_3) = (3,3,3)$, then $Y$ is $SU(2)$-abelian if and only if $|H_1(Y;\Z)|$ is even.
\end{proposition}

\begin{proof}
Lemma~\ref{lem:homology-sfs} tells us that $|H_1(Y;\Z)| = 9 \left| \sum_{i=1}^3 \beta_i \right|$.  Thus $|H_1(Y;\Z)|$ is even (which includes the case $\sum_i \beta_i = 0$, where $H_1(Y)$ is infinite) if and only if the number of $\beta_i$ which are odd is exactly $0$ or $2$.  In any case, some two of the $\beta_i$ must have the same parity, so we will assume without loss of generality that $\beta_1 \equiv \beta_2 \pmod{2}$.

Given a non-abelian representation $\rho: \pi_1(Y) \to SU(2)$, we write
\[ \rho(c_i) = e^{v_i\theta_i}, \qquad 1 \leq i \leq 3 \]
where $0 \leq \theta_i \leq \pi$ and the $v_i$ are imaginary unit quaternions for each $i$.  Again by Theorem~\ref{thm:triangle-reps} and Lemma~\ref{lem:triangle-h-image} we must have $\rho(h)=-1$, so $\rho$ is a valid representation if and only if $\rho(c_1c_2c_3)=1$ is satisfied and so is (for each $i$)
\[ \rho(c_i)^3 = (-1)^{\beta_i}\ \Longleftrightarrow\ e^{v_i \cdot 3\theta_i} = e^{v_i\cdot \beta_i\pi} \ \Longleftrightarrow\ \theta_i = \frac{(\beta_i+2n_i)\pi}{3} \]
for some integers $n_i$.  For each $i$ we have $\theta_i \not\in \{0,\pi\}$ since $\rho$ is non-abelian, and so
\[ \theta_i = \begin{cases} \pi/3, & \beta_i\text{ odd} \\ 2\pi/3, & \beta_i\text{ even}. \end{cases} \]

We now observe that if all three of the $\beta_i$ are even, then $\theta_1+\theta_2+\theta_3 = 2\pi$; and if only one is even then it must be $\beta_3$ and so $\theta_1+\theta_2 = \frac{\pi}{3}+\frac{\pi}{3} = \frac{2\pi}{3} = \theta_3$.  Thus in either of these cases Lemma~\ref{lem:triangle} says that $\rho$ cannot be non-abelian after all.  On the other hand, if all three are odd or exactly one is odd then
\[ (\theta_1,\theta_2,\theta_3) = \left(\frac{\pi}{3},\frac{\pi}{3},\frac{\pi}{3}\right) \text{ or } \left(\frac{2\pi}{3},\frac{2\pi}{3},\frac{\pi}{3}\right), \]
and all of the inequalities \eqref{eq:theta-triangle} and \eqref{eq:theta-sum} are strictly satisfied, so a non-abelian $\rho$ does exist as claimed.  We conclude that such a $\rho$ exists if and only if exactly one or three of the $\beta_i$ are odd, or equivalently if and only if $|H_1(Y;\Z)|$ is odd.
\end{proof}

\subsection{Four or more singular fibers}

In this subsection we complete the proof of Theorem~\ref{thm:sfs-over-s2} by considering Seifert fibered manifolds
\begin{equation} \label{eq:Y-many-singular-fibers}
Y = S^2((\alpha_1,\beta_1),(\alpha_2,\beta_2),\dots,(\alpha_n,\beta_n))
\end{equation}
where $n \geq 4$ and $2 \leq \alpha_1 \leq \alpha_2 \leq \dots \leq \alpha_n$.

\begin{theorem} \label{thm:large-seifert-fibered}
Let $Y$ be a Seifert fiber space of the form \eqref{eq:Y-many-singular-fibers}. Then there is a representation
\[ \pi_1(Y) \to SU(2) \]
with non-abelian image.
\end{theorem}

\begin{proof}
Since $\pi_1(Y)$ surjects onto the orbifold fundamental group $\Delta(\alpha_1,\dots,\alpha_n)$, the result follows from Theorem~\ref{thm:triangle-reps} in all cases except where
\[ (\alpha_1,\dots,\alpha_{n-2},\alpha_{n-1},\alpha_n) = (2,\dots,2,p,q). \]
Assume from now on that we are in this case, and let
\begin{align*}
Y' &= S^2((\alpha_2,\beta_2),(\alpha_3,\beta_3),\dots,(\alpha_n,\beta_n)) \\
\pi_1(Y') &= \langle c'_2,\dots,c'_n,h' \mid [h',c'_i]=(c'_i)^{\alpha_i}(h')^{\beta_i}=1\ \forall i, c'_2\dots c'_n=1\rangle.
\end{align*}
We will proceed by induction on $n$.

We first observe, again by Theorem~\ref{thm:triangle-reps} and Lemma~\ref{lem:triangle-h-image}, that a non-abelian representation $\rho': \pi_1(Y') \to SU(2)$ must satisfy $\rho'(h')=-1$.  Now $\beta_2$ is odd since it is coprime to $\alpha_2=2$, so $\rho'(c'_2)^2 = -1$, and thus we can replace $\rho'$ by a conjugate to ensure that $\rho'(c'_2)=i$.

If we have such a $\rho'$, we can now construct $\rho: \pi_1(Y) \to SU(2)$ by
\begin{align*}
\rho(c_1) &= j & \rho(c_2) &= k \\
\rho(h) &= -1 & \rho(c_i) &= \rho'(c'_i), \quad i \geq 3.
\end{align*}
Clearly the image is non-abelian, with $\rho(h)$ central.  It satisfies $\rho(c_i)^{\alpha_i} \rho(h)^{\beta_i} = 1$ for $i=1,2$ because $\alpha_1=\alpha_2=2$ implies that $\beta_1$ and $\beta_2$ are odd, and for $i \geq 3$ because the corresponding relations hold for $\rho'$.  Moreover, we have
\[ \rho(c_1\dots c_n) = j\cdot k \cdot \rho'(c'_3\dots c'_n) = i \cdot \rho'(c'_2)^{-1} = 1 \]
and so $\rho$ is indeed a representation of $\pi_1(Y)$.

For the base case $n=4$, we apply Proposition~\ref{prop:some-alpha-is-2} to get a non-abelian representation $\pi_1(Y') \to SU(2)$, and hence $\pi_1(Y) \to SU(2)$ by the above argument, in all cases except $(\alpha_1,\alpha_2,\alpha_3,\alpha_4)=(2,2,4,4)$.  In this last case, we construct the representation by hand: all of the $\beta_i$ must be odd, and we have
\[ \pi_1(Y) = \langle c_1,c_2,c_3,c_4,h \mid [h,c_i] = c_i^{\alpha_i}h^{\beta_i} = 1\ \forall i, c_1c_2c_3c_4=1 \rangle, \]
so we define $\rho: \pi_1(Y) \to SU(2)$ by letting $\rho(h)=-1$ and
\begin{align*}
\rho(c_1) &= i, & \rho(c_3) &= e^{\pi j/4}, \\
\rho(c_2) &= -i, & \rho(c_4) &= e^{-\pi j/4}.
\end{align*}
This completes the base case, and then for all $n \geq 5$ we have a non-abelian representation of $\pi_1(Y')$, so we get one for $\pi_1(Y)$ as well and the theorem follows by induction.
\end{proof}

\section{Seifert fibrations with other bases} \label{sec:sfs-over-other-bases}

We have so far restricted our attention to Seifert fibrations with base $S^2$, but we can in fact consider arbitrary closed surfaces.  In this section we will prove the following.

\begin{theorem} \label{thm:sfs-over-other-bases}
Let $Y$ be a Seifert fiber space whose base orbifold is not $S^2$ with any number of singular fibers.  Then $Y$ is $SU(2)$-abelian if and only if $Y$ is a lens space, $\RP^3\#\RP^3$, or a circle bundle over $T^2$ with even Euler number.
\end{theorem}

We begin by ruling out all possible bases other than $\RP^2$ and $T^2$.  In the following subsections we will show (Proposition~\ref{prop:sfs-over-t2}) that the $SU(2)$-abelian manifolds with base $T^2$ are precisely the claimed circle bundles, and then (Proposition~\ref{prop:sfs-over-rp2}) that the ones with base $\RP^2$ are either lens spaces or $\RP^3\#\RP^3$. 

\begin{proposition} \label{prop:possible-bases}
Let $Y$ be an $SU(2)$-abelian Seifert fibered space.  Then the base $B$ is either $S^2$, $\RP^2$, or $T^2$.
\end{proposition}

\begin{proof}
We observe that $\pi_1(Y)$ surjects onto the orbifold fundamental group of $B$.  If $B$ is an orientable surface of genus $g \geq 2$ then this further surjects onto 
\[ \pi_1(\Sigma_g) = \langle a_1,b_1,\dots,a_g,b_g \mid [a_1,b_1]\dots[a_g,b_g] = 1 \rangle, \]
and we have a representation which sends $(a_1,b_1,a_2,b_2) \mapsto (i,j,i,j)$ and all other generators to $1$.  Thus if $B$ is orientable then it must be $S^2$ or $T^2$.

If instead $B$ is non-orientable, we have $B = \#^g \RP^2$ for some $g \geq 1$.  Writing
\[ Y = B((\alpha_1,\beta_1),\dots,(\alpha_n,\beta_n)) \]
for some $n \geq 0$, we have a presentation
\begin{multline} \label{eq:pi1-nonorientable}
\pi_1(Y) = \langle a_1,\dots,a_g,c_1,\dots,c_n,h \mid a_i^{-1}ha_i=h^{-1}\ \forall i, \\
[h,c_j]=c_j^{\alpha_j}h^{\beta_j}=1\ \forall j,\ c_1\dots c_n a_1^2\dots a_g^2 = 1 \rangle
\end{multline}
from \cite[Theorem~6.1]{jankins-neumann}.  If $g \geq 2$ then we can define a non-abelian representation $ \pi_1(Y) \to SU(2)$ by sending $a_1 \mapsto i$, $a_2 \mapsto j$, and all other generators to $1$, so for non-orientable $B$ we conclude that $Y$ can only be $SU(2)$-abelian if $B=\RP^2$.
\end{proof}

\subsection{Fibrations with base $T^2$}

In this subsection we consider Seifert fiber spaces of the form
\begin{equation} \label{eq:sfs-over-t2}
Y = T^2((\alpha_1,\beta_1),\dots,(\alpha_n,\beta_n)).
\end{equation}
Just as in the case where the base is $S^2$, we can replace a pair $(1,k),(\alpha,\beta)$ with $(\alpha,\beta+k\alpha)$ without changing $Y$, so we can assume that either $n \leq 1$ or $\alpha_i \geq 2$ for all $i$.  If $n=0$ then $Y=T^3$ and $\pi_1(Y)$ is abelian, so we are also free to assume that $n \geq 1$ from now on.

\begin{proposition} \label{prop:sfs-over-t2}
Suppose that $Y$ is a Seifert fiber space over $T^2$ as in \eqref{eq:sfs-over-t2}, with $n \geq 1$ and with $\alpha_i \geq 2$ for all $i$ if $n \geq 2$.  Then $Y$ is $SU(2)$-abelian if and only if it is a circle bundle over $T^2$ with even Euler number.
\end{proposition}

\begin{proof}
We have a presentation
\[ \pi_1(Y) = \langle a,b,c_1,\dots,c_n,h \mid h\ \mathrm{central}, c_i^{\alpha_i}h^{\beta_i}=1\ \forall i, c_1\dots c_n[a,b] = 1\rangle, \]
so we attempt to define a representation $\rho: \pi_1(Y) \to SU(2)$ by setting
\[ \rho(h) = -1, \qquad \rho(c_j) = \cos\left(\tfrac{\beta_j}{\alpha_j}\pi\right) + v_j \sin\left(\tfrac{\beta_j}{\alpha_j}\pi\right) \]
for some unit vectors $v_1,\dots,v_n \in \R^3$.  This means that $\rho(h)$ will be central in the image of $\rho$ and that $\rho(c_j)^{\alpha_j} \rho(h)^{\beta_j} = 1$ for all $j$, so it remains to find elements $\rho(a), \rho(b) \in SU(2)$ such that
\[ [\rho(a),\rho(b)] = \left(\rho(c_1)\dots \rho(c_n)\right)^{-1}. \]
But this is possible because every element of $SU(2)$ is a commutator: we simply verify the identity
\[ e^{i\theta} = e^{i \theta/2} j (e^{i\theta/2})^{-1} j^{-1} \]
and note that every element of $SU(2)$ is conjugate to $e^{i\theta}$ for some $\theta$.

Assuming that $n \geq 2$ and $\alpha_k \geq 2$ for all $k$, we can now arrange for $\rho$ to have nonabelian image by taking $v_1=i$ and $v_2=j$, and choosing the rest of $v_3,\dots,v_n$ arbitrarily.  Then $\rho(c_1)$ and $\rho(c_2)$ do not commute since $\frac{\beta_1}{\alpha_1}\pi, \frac{\beta_2}{\alpha_2}\pi \not\in \pi\Z$.

If instead $n=1$, then $Y = T^2((\alpha_1,\beta_1))$ and we cannot necessarily guarantee that $\alpha_1 \neq 1$ as above.  However, if we can arrange that $\rho(c_1) \neq 1$, then the relation $c_1[a,b]=1$ implies that $[\rho(a),\rho(b)] \neq 1$, and hence that $\rho(a)$ and $\rho(b)$ do not commute.  We thus take $v_1=i$, so that
\[ \rho(c_1) = e^{i\pi \cdot \beta_j/\alpha_j}, \]
and this is not $1$ unless $\alpha_1=1$ and $\beta_1$ is even.  Thus if $n=1$ then $Y$ is not $SU(2)$-abelian except possibly when $Y=T^2((1,\beta))$ with $\beta$ even, in which case $Y$ is a circle bundle over $T^2$ with Euler number $-\beta$.  But then
\begin{align*}
\pi_1(Y) &= \langle a,b,c,h \mid h\ \mathrm{central},\ ch^{\beta}=1,\ c[a,b]=1 \rangle \\
&= \langle a,b,h \mid [a,h]=[b,h]=1,\ [a,b]=h^\beta\rangle.
\end{align*}
Since $h$ is central, a non-abelian representation $\rho: \pi_1(Y) \to SU(2)$ must send $h$ to $\pm 1$, hence $[a,b]$ to $(\pm 1)^{\beta}=1$.  But then $\rho(a)$ and $\rho(b)$ commute with each other and with $\rho(h)$, so $\rho$ has abelian image after all, contradiction.
\end{proof}

We remark that the circle bundles $Y \to T^2$ with nonzero, even Euler number have appeared before in a closely related context: Morgan, Mrowka, and Ruberman \cite{mmr} observed that there are no irreducible representations $\pi_1(Y) \to SU(2)$ and used this fact to prove a vanishing theorem for the Donaldson invariants of smooth 4-manifolds separated by $Y$.

\subsection{Fibrations with base $\RP^2$}

We now consider Seifert fiber spaces of the form
\[ Y = \RP^2((\alpha_1,\beta_1),\dots,(\alpha_n,\beta_n)). \]
If $n \leq 1$ then the base orbifold has a double cover of the form
\[ S^2 \to \RP^2 \quad\mathrm{or}\quad S^2(\alpha_1,\alpha_1) \to \RP^2(\alpha_1), \]
so pulling back the Seifert fibration along this covering gives a double cover $\tilde{Y}$ of $Y$ which is either a lens space, $S^3$, or $S^1\times S^2$.  If $\tilde{Y} \cong S^1\times S^2$ then $Y$ admits $\mathbb{S}^2\times\R$ geometry, and the only two such orientable 3-manifolds are $S^1\times S^2$ and $\RP^3\#\RP^3$ (see \cite[\S 4]{scott}, and note that $S^2 \tilde\times S^1$ and $\RP^2\times S^1$ are non-orientable).  We will see in Corollary~\ref{cor:lens-space-rp2} that in fact $S^1\times S^2$ has no Seifert fibration with base $\RP^2$, so then $Y=\RP^3\#\RP^3$.

In the case $\tilde{Y} \not\cong S^1\times S^2$, we know that $Y$ is either a lens space (but not $S^3$, since it has a nontrivial double cover) or a prism manifold.  Both of these admit Seifert fibrations with base $S^2$, and for prism manifolds the base orbifold is $S^2(2,2,n)$ for some $n \geq 2$, see \cite[Theorem~5.1]{jankins-neumann}.  Thus Theorem~\ref{thm:sfs-over-s2} tells us that the lens spaces are $SU(2)$-cyclic and the prism manifolds are not.

\begin{proposition} \label{prop:rp2-n-large}
Let $Y = \RP^2((\alpha_1,\beta_1),\dots,(\alpha_n,\beta_n))$ with $n \geq 2$ and $\alpha_i \geq 2$ for all $i$.  Then there is a representation $\pi_1(Y) \to SU(2)$ with non-abelian image.
\end{proposition}

\begin{proof}
From \cite[Theorem~6.1]{jankins-neumann} we have a presentation
\[ \pi_1(Y) = \langle a,c_1,\dots,c_n,h \mid a^{-1}ha=h^{-1}, [h,c_i]=c_i^{\alpha_i}h^{\beta_i}=1\ \forall i,c_1\dots c_na^2=1 \rangle. \]
We define a representation $\rho: \pi_1(Y) \to SU(2)$ by choosing purely imaginary unit quaternions $v_1,\dots,v_n$ and setting $\rho(h)=-1$ and
\[ \rho(c_k) = \cos\left(\frac{\beta_k}{\alpha_k}\pi\right) + \sin\left(\frac{\beta_k}{\alpha_k}\pi\right) v_k, \qquad 1 \leq k \leq n. \]
The product $(c_1\dots c_n)^{-1}$ has the form $\cos(\theta)+\sin(\theta)w$ for some purely imaginary $w$ and some angle $\theta$, so we use this to define
\[ \rho(a) = \cos\left(\frac{\theta}{2}\right) + \sin\left(\frac{\theta}{2}\right)w. \]
It is easy to check that $\rho$ satisfies all of the relations in the above presentation, so it is a well-defined representation regardless of our choices of $v_1,\dots,v_n$.  Moreover, none of the $\rho(c_k)$ can be equal to $\pm1$, since $\frac{\beta_k}{\alpha_k}\pi \not\in \pi\Z$.  Thus $\rho(c_i)$ and $\rho(c_j)$ commute if and only if $v_i = \pm v_j$, so for example we can take $v_1=i$ and $v_2=\dots=v_n=j$, and then $\rho$ will have nonabelian image.
\end{proof}

To summarize, we have proved the following.

\begin{proposition} \label{prop:sfs-over-rp2}
Let $Y$ be a Seifert fiber space with base $\RP^2$ and any number of singular fibers.  Then $Y$ is $SU(2)$-abelian if and only if $Y$ is a lens space or $\RP^3\#\RP^3$.
\end{proposition}

We can also determine which of these are lens spaces, which will be useful in Section~\ref{sec:four-fillings}.

\begin{corollary} \label{cor:lens-space-rp2}
The Seifert fiber space $Y=\RP^2((\alpha_1,\beta_1),\dots,(\alpha_n,\beta_n))$ has cyclic fundamental group if and only if $n=1$ and $\beta_1=\pm1$, and in this case $H_1(Y;\Z) = \Z/4\alpha_1\Z$.
\end{corollary}

\begin{proof}
If $\pi_1(Y)$ is cyclic then Proposition~\ref{prop:rp2-n-large} tells us that we can take $n < 2$, and we must have $n \neq 0$ or else $Y=\RP^2\times S^1$, so $n=1$.  We will drop the subscripts and write $Y=\RP^2((\alpha,\beta))$ from now on, with $\alpha \geq 1$ and $\beta$ relatively prime to $\alpha$.

The fundamental group of $Y$ has presentation
\[ \pi_1(Y) = \langle a,c,h \mid a^{-1}ha = h^{-1},\ [h,c]=c^{\alpha} h^{\beta}=1,\ ca^2=1\rangle. \]
The quotient $\pi_1(Y)/\llangle c\rrangle$ is the dihedral group of order $2|\beta|$, and if $\pi_1(Y)$ is cyclic then so is this quotient, so we must have $\beta=\pm1$.  In this case $c=a^{-2}$ and $h=c^{\mp \alpha} = a^{\pm 2\alpha}$, so $\pi_1(Y)$ is generated by $a$; the remaining relation $a^{-1}ha=h^{-1}$ becomes $h^2=1$, or $a^{4\alpha}=1$, so $a$ has order $4\alpha$.
\end{proof}

\section{$\mathrm{Sol}$ manifolds} \label{sec:sol}

Every Seifert fibered 3-manifold admits a geometric structure, and in fact the geometric structure is determined by the orbifold Euler characteristic of the base (which can be positive, zero, or negative) and the Euler number of the Seifert fibration, as in \cite[Theorem~5.3]{scott}.  We can thus determine the geometries of the $SU(2)$-abelian Seifert fibered manifolds in Theorem~\ref{thm:su2-abelian-sfs}:

\begin{itemize}
\item $S^3$ and the lens spaces have $\mathbb{S}^3$ geometry.
\item $S^1\times S^2$ and $\RP^3\#\RP^3$ have $\mathbb{S}^2\times\R$ geometry.
\item $T^3$ and the Seifert fibered spaces
\[ S^2((3,1),(3,1),(3,-2)) \quad\mathrm{and}\quad S^2((2,1),(4,1),(4,-3)), \]
which up to orientation are the unique ones over their base orbifolds with Euler number 0, have $\mathbb{E}^3$ geometry.
\item All of the remaining $SU(2)$-abelian Seifert fibered spaces have $\mathrm{Nil}$ geometry.
\end{itemize}
The six geometries with Seifert fibered representatives consist entirely of Seifert fibered manifolds, so there are no closed, orientable, $SU(2)$-abelian 3-manifolds with either $\mathbb{H}^2\times \R$ or $\widetilde{SL(2,\R)}$ geometry.

The two geometries which do not allow Seifert fibrations are $\mathbb{H}^3$ and $\mathrm{Sol}$.  We cannot say much about $\mathbb{H}^3$, but in fact we can understand the $\mathrm{Sol}$ case completely, and this is the goal of this section.

\begin{theorem} \label{thm:count-sol-bundles}
A closed, orientable 3-manifold with $\mathrm{Sol}$ geometry is $SU(2)$-abelian if and only if it is the mapping torus of a diffeomorphism $\phi: T^2 \to T^2$, represented by a matrix $\phi \in SL(2,\Z)$ which is conjugate in $GL(2,\Z)$ to one of
\[ \begin{pmatrix} -3 & -1 \\ 1 & 0 \end{pmatrix},\ \begin{pmatrix} -3 & 1 \\ 2 & -1 \end{pmatrix},\ \mathrm{or}\ \begin{pmatrix} -3 & 4 \\ 2 & -3 \end{pmatrix}. \]
Since the mapping torus depends only on the conjugacy class, there are thus three such $SU(2)$-abelian manifolds.
\end{theorem}

\begin{proof}
If we let $N$ denote the twisted $I$-bundle over the Klein bottle $K$ with boundary $T^2$, then by \cite[Theorem~5.3]{scott}, any closed, orientable $\mathrm{Sol}$ 3-manifold is one of the following:
\begin{itemize}
\item a union $N \cup_\phi N$ of two copies of $N$, glued by some diffeomorphism $\phi: T^2 \to T^2$ of their boundaries; or
\item a $T^2$-bundle over $S^1$ with hyperbolic monodromy $\phi \in SL(2,\Z)$, meaning that $\phi$ has distinct real eigenvalues, i.e.\ $|{\tr}(\phi)|>2$.
\end{itemize}
We prove in Proposition~\ref{prop:sol-twisted} that none of the examples in the first case are $SU(2)$-abelian, and we characterize the $SU(2)$-abelian torus bundles in Proposition~\ref{prop:sol-torus-bundle} in terms of their monodromies in $SL(2,\Z)$.  Finally, any two monodromies have the same mapping torus if they are conjugate in $GL(2,\Z)$, so in Proposition~\ref{prop:count-quadratic-forms} we determine the number of such conjugacy classes.
\end{proof}

\begin{proposition} \label{prop:sol-twisted}
If $Y$ is a union of two twisted $I$-bundles over the Klein bottle, glued along a diffeomorphism of their boundaries, then $\pi_1(Y)$ surjects onto the quaternion group $Q$ of order 8.  In particular, $Y$ is not $SU(2)$-abelian.
\end{proposition}

\begin{proof}
The Klein bottle has fundamental group
\[ \pi_1(K) = \langle a,b \mid bab^{-1} = a^{-1} \rangle, \]
and the torus $T = \partial N$ is a double cover of $K$ whose fundamental group is both abelian and an index-2 subgroup of $\pi_1(K)$, hence the kernel of a surjection $\pi_1(K) \to \Z/2\Z$.  We use the representation
\[ \rho: \pi_1(K) \to SU(2), \qquad \rho(a) = e^{i\pi/4}, \quad \rho(b) = j \]
to see that $a^2$ does not commute with either $b$ or $ab$, since their images under $\rho$ do not commute.  But then the surjection $\pi_1(K) \to \Z/2\Z$ cannot send $a \mapsto 1$: if it did, the kernel $\pi_1(T)$ would contain both $a^2$ and either $b$ or $ab$ depending on whether $b\mapsto 0$ or $b\mapsto 1$.  It must therefore send $a \mapsto 0$ and (since it is surjective) $b\mapsto 1$, and hence $\pi_1(T) = \langle a,b^2 \rangle \subset \pi_1(K)$.

Now write $Y = N_1 \cup_\phi N_2$, and give the homology of $T_i = \partial N_i$ a basis of curves $[a_i],[b_i^2]$ for $i=1,2$.  Suppose that in this basis the gluing diffeomorphism $\phi:T_1 \to T_2$ has the form
\[ \phi = \begin{pmatrix} m & n \\ p & q \end{pmatrix}, \]
meaning that $a_2 = a_1^m (b_1^2)^n$ and $b_2^2 = a_1^p (b_1^2)^q$ in $\pi_1(Y)$.  Since $\phi$ is invertible, it has determinant $\pm1$, and in particular $p$ and $q$ are not both even.  We now define a homomorphism
\[ f: \pi_1(Y) = \pi_1(N_1) \ast_{\pi_1(T)} \pi_1(N_2) \twoheadrightarrow Q \]
by setting $f(b_1)=i$, $f(b_2)=j$, $f(a_1) = (-1)^{q-1}$, and $f(a_2) = f(a_1)^m(-1)^n$.  It is straightforward to check that these are well-defined on each $\pi_1(N_i)$, and that they satisfy $f(b_2^2) = f(a_1^pb_1^{2q})$ (since either $p$ or $q$ is odd) and $f(a_2) = f(a_1^m b_1^{2n})$.  Thus $f$ is well-defined, and it is surjective since $i$ and $j$ generate $Q$.
\end{proof}

Suppose instead that $Y$ is a mapping torus of some hyperbolic diffeomorphism $\phi: T^2 \to T^2$, which we identify with a matrix
\[ \phi = \begin{pmatrix} a & b \\ c & d \end{pmatrix} \in SL(2,\Z) \]
with $|{\tr}(\phi)| = |a+d| > 2$, where $\phi$ is written in terms of a basis $(x,y)$ of $\pi_1(T^2)$.   We will write $\tau=\tr(\phi)$, so that $|\tau| > 2$.  Then
\[ \pi_1(Y) = \langle x,y,t \mid [x,y]=1,\ t (x^ey^f) t^{-1} = \phi(x^ey^f) \ \forall e,f\rangle. \]
The following lemma applies in slightly more generality than for hyperbolic $\phi$, since the proof only assumes that $\tau \neq -2$.

\begin{lemma} \label{lem:torus-bundle-su2}
There is a non-abelian representation $\rho: \pi_1(Y) \to SU(2)$ if and only if one of the two representations defined by
\begin{align*}
\rho(x) &= e^{i\theta_1}, & \rho(y) &= e^{i\theta_2}, & \rho(t) &= j,
\end{align*}
where
\[ \begin{pmatrix} \theta_1 \\ \theta_2 \end{pmatrix} = \frac{2\pi}{\tau+2}\begin{pmatrix}d+1\\-b\end{pmatrix} \quad\mathrm{or}\quad \frac{2\pi}{\tau+2}\begin{pmatrix}-c\\a+1\end{pmatrix}, \]
has non-abelian image, or equivalently if and only if the corresponding values of $\theta_1$ and $\theta_2$ are not both integer multiples of $\pi$.
\end{lemma}

\begin{proof}
It is straightforward to check that these two choices of $(\theta_1,\theta_2)$ both define representations, since $\rho(x)$ and $\rho(y)$ clearly commute and since
\[ \rho(t) \rho(x^ey^f) \rho(t)^{-1} = j e^{i(e\theta_1+f\theta_2)} j^{-1} = e^{-i(e\theta_1+f\theta_2)} \]
is equal to
\[ \rho(x^{ae+bf}y^{ce+df}) = e^{i((a\theta_1+c\theta_2)e+(b\theta_1+d\theta_2)f)} \]
for all integers $e$ and $f$ if and only if $(a+1)\theta_1 + c\theta_2$ and $b\theta_1 + (d+1)\theta_2$ are both multiples of $2\pi$.  This last condition is easily verified for both values of $(\theta_1,\theta_2)$, since these quantities are equal to $2\pi\left(\frac{(a+1)(d+1)-bc}{\tau+2}\right) = 2\pi$ and $0$ in some order.

In general, since $x$ and $y$ commute, given any non-abelian representation $\rho: \pi_1(Y) \to SU(2)$ we can find angles $\theta_1,\theta_2$ such that 
\[ \rho(x) = e^{i\theta_1}, \qquad \rho(y) = e^{i\theta_2} \]
up to conjugacy, and then the relation $t (x^ey^f) t^{-1} = \phi(x^ey^f)$ becomes
\begin{equation} \label{eq:torus-bundle-t}
\rho(t) e^{i(e\theta_1+f\theta_2)} \rho(t)^{-1} = e^{i((ae+bf)\theta_1 + (ce+df)\theta_2)} \quad \mathrm{for\ all}\ e,f.
\end{equation}

We cannot have $\theta_1, \theta_2 \in \pi\Z$ because then $\rho(x)$ and $\rho(y)$ would both be $\pm 1$ and hence $\rho$ would have abelian image, regardless of the value of $\rho(t)$.  Supposing without loss of generality that $e^{i\theta_1} \neq \pm 1$, setting $(e,f)=(1,0)$ in \eqref{eq:torus-bundle-t} gives
\[ \rho(t) e^{i\theta_1} \rho(t)^{-1} = e^{i(a\theta_1+c\theta_2)}. \]
Since $\re(e^{i\theta_1})$ is invariant under conjugacy, the only elements of the form $e^{i\psi}$ conjugate to $e^{i\theta_1}$ are $e^{\pm i\theta_1}$, so the right side must be one of these.  If it is $e^{i\theta_1}$ then $\rho(t)$ commutes with $e^{i\theta_1}$, hence it must have the form $\rho(t)=e^{i\alpha}$ for some $\alpha$, contradicting the assumption that $\rho$ has non-abelian image.  Thus the right side must be $e^{-i\theta_1}$, and conjugating both sides by $-j$ gives
\[ (-j\rho(t))e^{i\theta_1}(-j\rho(t))^{-1} = (-j)e^{-i\theta_1}(-j)^{-1} = e^{i\theta_1}, \]
so that $-j\rho(t) = e^{i\alpha}$ for some $\alpha$ and thus $\rho(t) = je^{i\alpha}$.  In fact, any value of $\alpha$ leads to the same result, so we are free to take $\rho(t) = j$.

It now follows that $\rho(t)e^{i\beta}\rho(t)^{-1} = e^{-i\beta}$ for all $\beta$.  Applying this to \eqref{eq:torus-bundle-t}, we conclude just as before that $(a+1)\theta_1 + c\theta_2$ and $b\theta_1 + (d+1)\theta_2$ are both integer multiples of $2\pi$, or equivalently that
\[ \begin{pmatrix} a+1 & c \\ b & d+1 \end{pmatrix}\begin{pmatrix} \theta_1 \\ \theta_2 \end{pmatrix} = \begin{pmatrix} 2\pi m \\ 2\pi n \end{pmatrix} \]
for some integers $m$ and $n$.  The matrix on the left has determinant $\tau+2 \neq 0$, hence is invertible, so we set
\begin{align*}
\begin{pmatrix} \theta_1 \\ \theta_2 \end{pmatrix} &= \frac{1}{\tau+2}\begin{pmatrix} d+1 & -c \\ -b & a+1 \end{pmatrix} \begin{pmatrix} 2\pi m \\ 2\pi n \end{pmatrix} \\
& = \frac{2\pi}{\tau+2}\left[\begin{pmatrix} d+1 \\ -b \end{pmatrix}m + \begin{pmatrix} -c \\ a+1 \end{pmatrix}n \right].
\end{align*}
If $\frac{2}{\tau+2}\left(\begin{smallmatrix} d+1 \\ -b \end{smallmatrix}\right)$ and $\frac{2}{\tau+2}\left(\begin{smallmatrix} -c \\ a+1 \end{smallmatrix}\right)$ both have integral coordinates then we have $\theta_1,\theta_2 \in \pi\Z$ for all choices of $m$ and $n$, but then $\rho$ cannot be non-abelian for any $m$ and $n$.  Thus if any non-abelian $\rho$ exists, either $(m,n)=(1,0)$ or $(m,n)=(0,1)$ must yield such a $\rho$, and these give the values of $(\theta_1,\theta_2)$ in the statement of the lemma.
\end{proof}

\begin{proposition} \label{prop:sol-torus-bundle}
Let $Y$ be an orientable $\mathrm{Sol}$ manifold which is a $T^2$-bundle over $S^1$, whose monodromy is identified with an element $\phi \in SL(2,\Z)$.  Then $Y$ is $SU(2)$-abelian if and only if either $\phi$ has trace $-3$ or $-4$, or $\tr(\phi) = -6$ and $\phi \equiv \Id \pmod{2}$.
\end{proposition}

\begin{proof}
Again we write $\tau = \tr(\phi)$ with $|\tau| > 2$ and adopt the same notation as before.  According to Lemma~\ref{lem:torus-bundle-su2}, there is a non-abelian representation $\pi_1(Y) \to SU(2)$ if and only if either $2\left(\begin{smallmatrix} d+1 \\ -b \end{smallmatrix}\right)$ or $2\left(\begin{smallmatrix} -c \\ a+1 \end{smallmatrix}\right)$ is not an integral multiple of $\tau+2$.  This is equivalent to
\[ 2\begin{pmatrix} a+1 & b \\ c & d+1 \end{pmatrix} = 2(\phi+\Id) \]
not being $\tau+2$ times an integral matrix.

If $2(\phi+\Id)$ is a multiple of $\tau+2$, then its determinant $4(\tau+2)$ is a multiple of $(\tau+2)^2$, so that $\tau+2$ divides $4$.  If $Y$ is $SU(2)$-abelian, then since $|\tau| > 2$ we must therefore have $\tau+2 \in \{-1,-2,-4\}$; when $\tau+2$ is $-1$ or $-2$ it is certainly the case that this matrix is a multiple of $\tau+2$, so that $Y$ is indeed $SU(2)$-abelian if $\tau$ is $-3$ or $-4$.  This leaves only the case $\tau+2 = -4$, for which $2(\phi+\Id)$ is a multiple of $-4$ if and only if $\phi \equiv \Id \pmod{2}$.
\end{proof}

A torus bundle over $S^1$ is not determined uniquely up to homeomorphism by the monodromy $\phi \in SL(2,\Z)$: any two choices which are conjugate by an element of $GL(2,\Z)$ produce the same 3-manifold.  Thus Proposition~\ref{prop:sol-torus-bundle} provides relatively few $SU(2)$-abelian manifolds.

To elaborate on this point, the number $c(\tau)$ of conjugacy classes in $SL(2,\Z)$ with fixed trace $\tau$ is well-understood.  For example, Chowla, Cowles, and Cowles \cite{ccc} prove for all $\tau \neq \pm2$ that $c(\tau)$ (which they denote $\bar{H}(\tau)$) is equal to the number $\bar{h}(\tau^2-4)$ of equivalence classes of binary quadratic forms
\[ Q(x,y) = ax^2+bxy+cy^2 \]
with discriminant $b^2-4ac = \tau^2-4$.  Here two quadratic forms $Q$ and $Q'$ are said to be equivalent if there are integers $p,q,r,s$ with $ps-qr=1$ such that
\[ Q(px+qy, rx+sy) = Q'(x,y). \]
The bijection between conjugacy classes in $SL(2,\Z)$ and equivalence classes of forms is given by
\[ A = \begin{pmatrix} a & b \\c & d \end{pmatrix} \quad\longleftrightarrow\quad Q_A(x,y) = bx^2 + (d-a)xy - cy^2. \]
Note that $Q_A$ has discriminant $(d-a)^2+4bc = (a+d)^2 - 4(ad-bc) = \tau^2 - 4$.

\begin{proposition} \label{prop:count-quadratic-forms}
Let $A \in SL(2,\Z)$, and define $A_{-3}, A_{-4}, A_{-6} \in SL(2,\Z)$ by
\[ A_{-3} = \begin{pmatrix} -3 & -1 \\ 1 & 0 \end{pmatrix}, \quad A_{-4} = \begin{pmatrix} -3 & 1 \\ 2 & -1 \end{pmatrix}, \quad A_{-6} = \begin{pmatrix} -3 & 4 \\ 2 & -3 \end{pmatrix}. \]
If $\tr(A)=-3$ then $A$ is conjugate in $SL(2,\Z)$ to $A_{-3}$.  If $\tr(A)=-4$ then $A$ is conjugate in $SL(2,\Z)$ to exactly one of $A_{-4}$ and its transpose, and these are conjugate to each other in $GL(2,\Z)$.  If $\tr(A) = -6$ and $A \equiv \Id \pmod{2}$ then $A$ is conjugate in $SL(2,\Z)$ to $A_{-6}$.
\end{proposition}

\begin{proof}
Write $\tau = \tr(A)$.  By the above argument, in the cases $\tau=-3$ and $\tau=-4$, we wish to count quadratic forms of discriminant $5$ and $12$ respectively.  For $\tau=-6$, the discriminant is $32$, and the condition $A \equiv \Id \pmod{2}$ is equivalent to the coefficients of $Q_A$ all being even.  (One direction is immediate; for the converse, if $Q(x,y)=px^2+qxy+ry^2$ with $q^2-4pr=32$ and $p,q,r$ even then $q$ must be a multiple of 4, and the bijection identifies $Q$ with the matrix $\left(\begin{smallmatrix} -3-q/2 & p \\ -r & -3+q/2 \end{smallmatrix}\right)$, which reduces mod 2 to the identity matrix.)  The number of such conjugacy classes in $SL(2,\Z)$ is therefore the number of equivalence classes of quadratic forms of discriminant $8$, by the bijection $A \leftrightarrow \frac{1}{2}Q_A$.  In other words, the numbers of conjugacy classes we wish to count for trace $-3$, $-4$, and $-6$ are given by $\bar{h}(5)$, $\bar{h}(12)$, and $\bar{h}(8)$ respectively.

These numbers of quadratic forms up to equivalence have been heavily studied by number theorists, see e.g.\ \cite{buell}; the number of equivalence classes of \emph{primitive} quadratic forms of discriminant $D > 0$ is the order of the narrow class group of $\Q(\sqrt{D})$ \cite[Theorem~6.20]{buell}.  In the cases $D=5,8,12$, all binary quadratic forms are in fact primitive.  Indeed, if $Q$ is a multiple of $n \geq 2$ then $\frac{1}{n}Q$ has discriminant $\frac{D}{n^2}$; we note that $D=5$ is square-free, and that $D=8,12$ are both multiples of $2^2$ but with $\frac{D}{2^2}$ either $2$ or $3\pmod{4}$, hence not a discriminant.  These numbers are catalogued for small $D$ in \cite[Appendix~2]{buell}, where we see that $\bar{h}(5)=\bar{h}(8)=1$ and $\bar{h}(12)=2$.

In conclusion, every element of $A\in SL(2,\Z)$ with $\tr(A)=-3$, or with $\tr(A)=-6$ and $A \equiv \Id \pmod{2}$, is conjugate to $A_{-3}$ or to $A_{-6}$ respectively.  There are two conjugacy classes of matrices with trace $-4$, which we claim are represented by $A_{-4}$ and by $(A_{-4})^T$.  To see that these are not conjugate in $SL(2,\Z)$, note that otherwise they satisfy $MA_{-4} = (A_{-4})^TM$ where $M$ is an integer matrix, and this gives a rank-2 system of linear equations in the entries of $M$ whose solutions are precisely $M=\left(\begin{smallmatrix} 2c+2d & -c \\ -c & d \end{smallmatrix}\right)$ for integers $c$ and $d$.  Then $\det(M) = 3d^2 - (c-d)^2$, and this is never $1\pmod{4}$, so we cannot have $M \in SL(2,\Z)$.  However, taking $(c,d)=(1,0)$, and hence $M=\left(\begin{smallmatrix} 2 & -1 \\ -1 & 0 \end{smallmatrix}\right)$, shows that $A_{-4}$ is conjugate to $(A_{-4})^T$ in $GL(2,\Z)$.
\end{proof}

\section{Hyperbolic manifolds} \label{sec:hyperbolic}

\subsection{Some hyperbolic examples} \label{ssec:double-covers}

Hyperbolic manifolds are the most mysterious geometric manifolds from the perspective of $SU(2)$ character varieties.   A theorem of Thurston \cite[Proposition~3.1.1]{culler-shalen-splittings} says that they are never $SL(2,\C)$-cyclic, because the canonical representations of their fundamental groups into $PSL(2,\C)$ lift to faithful $SL(2,\C)$ representations, but nonetheless it is still possible for them to be $SU(2)$-cyclic.  The only examples we currently know are produced by Theorem~\ref{thm: bridge number 3}; here we describe their construction, as branched double covers of certain knots in $S^3$.

\begin{definition}[\cite{zentner-simple}] \label{def:su2-simple}
A representation $G \to SU(2)$ is \emph{binary dihedral} if it is conjugate to a representation whose image lies in the binary dihedral group
\[ \{ e^{i\theta} \} \sqcup \{e^{i\theta}j\} \subset SU(2). \]
A knot $K \subset S^3$ is \emph{$SU(2)$-simple} if every representation
\[ \pi_1(S^3 \setminus N(K)) \to SU(2) \]
which sends a meridian $\mu$ of $K$ to a purely imaginary unit quaternion (equivalently, a matrix of trace zero) is binary dihedral.
\end{definition}

If a branched double cover $\dcover(K)$ is $SU(2)$-cyclic, then the knot $K$ is $SU(2)$-simple; see \cite[\S 3]{zentner-simple} for details.  We will use results of Cornwell \cite{cornwell} and the second author \cite{zentner-simple} to prove a partial converse, namely that the branched double cover of an $SU(2)$-simple 3-bridge knot is $SU(2)$-cyclic.  This was implicitly stated in \cite[\S 10]{zentner-simple} as a consequence of \cite[Theorem~1.1]{cornwell}, but the proof is not at all obvious so we provide a complete argument here.

We first recall that we have a short exact sequence of groups
\begin{equation}\label{eq:ses branched cover} 
1 \to \pi_1(\dcover(K)) \to \pi_1(S^3 \setminus N(K))/ \llangle \mu^2 \rrangle \to \Z/2\Z \to 1
\end{equation}
which is split by sending the generator of $\Z/2\Z$ to a meridian $\mu$.  We also recall that the adjoint action of $SU(2)$ on its Lie algebra gives a double cover $SU(2) \to SO(3)$; then the binary dihedral group is the preimage of the dihedral group, which is the subgroup of $SO(3)$ isomorphic to $O(2) \cong SO(2) \rtimes \Z/2\Z$ consisting of rotations around the $i$-axis and reflections along lines through the origin in the $jk$-plane.
 
Let $\rho \colon \pi_1(S^3 \setminus N(K)) \to SU(2)$ be a representation which sends a meridian $\mu$ of the knot $K$ to an imaginary quaternion. Then $\rho(\mu)^2 = -1$, so $\rho$ gives rise to an $SO(3)$ representation which sends the meridian $\mu$ to an element of order $2$. Therefore $\rho$ descends to a representation
\[ \overline{\rho}: \pi_1(S^3 \setminus N(K))/\llangle \mu^2 \rrangle \to SO(3), \]
which by restriction also defines a representation 
\[ \overline{\overline{\rho}}\colon \pi_1(\dcover(K)) \to SO(3). \]
As explained for instance in \cite[\S 3]{zentner-simple}, this restriction $\overline{\overline{\rho}}$ is cyclic if the representation $\rho$ is binary-dihedral.

It turns out that the converse also holds: If a knot $K$ is $SU(2)$-simple, then all representations $\overline{\overline{\rho}}\colon \pi_1(\dcover(K)) \to SU(2)$ which arise in this way from a representation $\rho$ of the knot complement have cyclic image. This follows essentially from the proof of Proposition~\ref{prop: SU(2)-cyclic} below. It is not just a group theoretic argument as one might expect (two semi-direct products with $\Z/2\Z$ occur), but it uses the fact that $H_1(\dcover(K))$ has odd order. 
	
Moreover, under certain circumstances, the representations $\overline{\overline{\rho}}$ arising this way are the only representations of $\pi_1(\dcover(K))$ that can occur.  For instance, we will see below that this can follow from results of Cornwell \cite{cornwell}. To explain this, we use the following simple lemma. 

\begin{lemma}\label{le:equivariant}
Let $H$ be a group, and let $\tau \in \Aut(H)$ be an element of order dividing $k$, which then gives rise to a homomorphism $\Z/k\Z \to \Aut(H)$ and a semi-direct product $H \rtimes_{\tau} \Z/k\Z$ with group structure
\[ (a,h) \cdot (b,l) := (a \tau^h(b), h+l). \]
Suppose $\varphi\colon H \to SO(3)$ is a representation which is $\Z/k\Z$-equivariant in the following sense: There exists an element $A \in SO(3)$ of order dividing the order of $\tau$ such that we have
\[ \varphi(\tau(a)) = A \varphi(a) A^{-1} \]
for all $a \in H$.  Then there is a representation
\[ \psi\colon H \rtimes_{\tau} \Z/k\Z \to SO(3) \]
which extends $\varphi$, in the sense that $\psi(a,0) = \varphi(a)$ for all $a \in H$.
\end{lemma}

\begin{proof}
We simply define $\psi((a,h)):= \varphi(a) A^{h}$. Then the equivariance of $\varphi$ implies that we have
\begin{align*}
\psi((a \, \tau^{h}(b), h+l)) &= \varphi(a) \, \varphi(\tau^{h}(b))\, A^{h+l} \\
&= \varphi(a) \cdot A^h \varphi(b) A^{-h}\cdot A^{h+l} \\
&= \varphi(a) A^h \cdot \varphi(b) A^l = \psi((a,h)) \cdot \psi((b,l))
\end{align*}
for all $(a,h), (b,l) \in H \rtimes_{\tau} \Z/k\Z$. 
\end{proof}

In the short exact sequence \eqref{eq:ses branched cover} above, the group $\pi_1(S^3 \setminus N(K))/\llangle \mu^2 \rrangle$ is a semi-direct product
\begin{equation} \label{eq:sigma-semidirect}
\pi_1(S^3 \setminus N(K))/\llangle \mu^2 \rrangle \cong \pi_1(\dcover(K)) \rtimes_\tau \Z/2\Z,
\end{equation}
with $\tau$ being the involution which defines the branched double cover $\dcover(K)$ of $K$.

We now recall some useful facts from \cite[Lemma~3.2]{zentner-simple}.  The group $H_1(\dcover(K))$ has order $\det(K) = |\Delta_K(-1)|$, which is odd, so $\dcover(K)$ is a $\Z/2\Z$-homology sphere.  Every representation $\pi_1(\dcover(K))\to SO(3)$ lifts to $SU(2)$ because the obstruction to doing so lives in $H^2(\dcover(K);\Z/2\Z)=0$, so the map
\begin{equation} \label{eq:sigma-su2-so3}
\Hom(\pi_1(\dcover(K)), SU(2)) \to \Hom(\pi_1(\dcover(K)), SO(3))
\end{equation}
induced by composition with the covering map $SU(2)\to SO(3)$ is surjective.  It is also injective because any two $SU(2)$-representations with the same image agree up to multiplication by a character $\pi_1(\dcover(K)) \to \{\pm 1\}$, which factors through the abelianization $H_1(\dcover(K))$ and is thus trivial.  Moreover, any abelian $SO(3)$-representation factors through $H_1(\dcover(K))$ and thus has image a finite, odd-order subgroup of $SO(3)$; these are all cyclic, generated by a rotation of odd order about some fixed axis, and their preimages in $SU(2)$ are also cyclic.  Thus the bijection \eqref{eq:sigma-su2-so3} takes cyclic representations to cyclic ones and non-abelian representations to non-abelian ones.

\begin{proposition}\label{prop: SU(2)-cyclic}
If $K$ is an $SU(2)$-simple knot, and if the representation
\[ \varphi\colon \pi_1(\dcover(K)) \to SO(3) \]
is $\Z/2\Z$-equivariant in the sense of Lemma~\ref{le:equivariant} above, then its lift to
\[ \hat\varphi: \pi_1(\dcover(K)) \to SU(2) \]
has cyclic image.
\end{proposition}
\begin{proof}
By Lemma~\ref{le:equivariant} above and the identification \eqref{eq:sigma-semidirect}, the representation $\varphi$ extends to a representation
\[ \psi\colon \pi_1(S^3 \setminus N(K))/\llangle \mu^2 \rrangle \to SO(3). \]
As $K$ is $SU(2)$-simple, the group $\pi_1(S^3\setminus N(K))/\llangle \mu^2 \rrangle$ has only dihedral representations in $SO(3)$, and hence $\psi$ is conjugate to a representation mapping into the dihedral group $O(2)$ as defined above.  
 
We claim that $\psi$, when restricted to $\pi_1(\dcover(K))$, must have image in $SO(2) \subseteq O(2) \subseteq SO(3)$.  Replacing $\psi$ with a conjugate whose image lies in $O(2)$, we consider the composition
\[ \pi_1(\dcover(K)) \xrightarrow{\psi} O(2) \xrightarrow{\det} \{\pm1\}. \]
As a homomorphism to an abelian group, it factors through $H_1(\dcover(K))$.  For a knot $K$ the latter group has order $\det(K) = |\Delta_K(-1)|$, which is odd.  Therefore the image of the composition $\det \circ \psi|_{\pi_1(\dcover(K))}$ must be trivial, and this means that $\varphi = \psi|_{\pi_1(\dcover(K))}$ has image in $SO(2)$, which is abelian.  But we have seen in the discussion preceding this proposition that the lift $\hat\varphi$ to $SU(2)$ must therefore have cyclic image, as desired.
\end{proof}

\begin{proposition}\label{prop: su2-cyclic branched double cover}
Suppose $K$ is an $SU(2)$-simple knot, and suppose also that the $SL(2,\C)$-character variety of the branched double cover $\dcover(K)$ consists entirely of characters which are invariant under the involution $\tau$.  Then $\dcover(K)$ is $SU(2)$-cyclic. 	
\end{proposition}

\begin{proof}
We wish to show that every representation $\pi_1(\dcover(K)) \to SU(2)$ has cyclic image.  Supposing otherwise, let
\[ \hat{\varphi}: \pi_1(\dcover(K)) \to SU(2) \]
be a representation with non-cyclic image, which is equivalent to $\hat{\varphi}$ being irreducible, and let $\varphi: \pi_1(\dcover(K)) \to SO(3)$ be its image under the bijection \eqref{eq:sigma-su2-so3}.  We claim that $\varphi$ is $\Z/2\Z$-equivariant in the sense of Lemma~\ref{le:equivariant}, at which point the assumption of $SU(2)$-simplicity together with Proposition~\ref{prop: SU(2)-cyclic} tells us that $\hat\varphi$ has cyclic image and we have a contradiction.
		
Both of the irreducible representations
\[ \hat\varphi,\ \hat\varphi\circ\tau: \pi_1(\dcover(K)) \to SU(2) \subset SL(2,\C) \]
have the same characters by assumption. By \cite[Proposition~1.5.2]{culler-shalen-splittings}, these two representations are conjugate: there exists a matrix $\hat{A} \in SL(2,\C)$ such that 
\begin{equation}\label{eq:equivariance varphi}
\hat{\varphi} \circ \tau = \hat{A} \, \hat{\varphi} \, \hat{A}^{-1}.
\end{equation}
Since $\hat\varphi$ and $\hat\varphi \circ \tau$ are irreducible $SU(2)$ representations which are conjugate in $SL(2,\C)$, an argument of Klassen \cite[Proof of Proposition~15]{klassen} says that we can take $\hat{A}$ to be an element of $SU(2)$ as well.  Moreover, as $\tau$ has order $2$, it follows that 
\[ \hat{A}^2 \hat{\varphi} \hat{A}^{-2} = \hat{\varphi}. \]
By Schur's lemma (which applies since $\hat{\varphi}$ is irreducible) it follows that $\hat{A}^2$ must be a multiple of the identity, hence $\hat{A}^2 = \pm 1$.

Letting $A \in SO(3)$ be the image of $\hat{A} \in SU(2)$, with $A^2 = \Id$, we now obtain
\[ \varphi \circ \tau = A \varphi A^{-1} \]
from equation~\eqref{eq:equivariance varphi} above.  This is the definition of $\Z/2\Z$-equivariance, so the proof of the claim is complete.
\end{proof}

Using a result of Cornwell \cite{cornwell}, we are now able to prove:

\begin{theorem}\label{thm: bridge number 3}
Suppose $K$ is an $SU(2)$-simple knot of bridge number $b(K) \leq 3$. Then the branched double cover $\dcover(K)$ is $SU(2)$-cyclic. 
\end{theorem}

\begin{proof}
By \cite[Theorem~1.1]{cornwell}, a knot with bridge number $b(K) \leq 3$ has the property that the $SL(2,\C)$-character variety $X(\dcover(K))$ consists entirely of characters which are invariant under the action of the involution $\tau$ of $\dcover(K)$. The conclusion therefore follows from Proposition \ref{prop: su2-cyclic branched double cover}.
\end{proof}

\begin{example} \label{ex:m036}
Cornwell \cite[\S 6]{cornwell} proves that $8_{18}$ is $SU(2)$-simple, and it has bridge index 3, so its branched double cover $\dcover(8_{18})$ is $SU(2)$-cyclic.  This 3-manifold was previously mentioned in \cite[\S 10]{zentner-simple}.  It is known to be the branched 4-fold cyclic cover of the figure-eight knot \cite{vesnin-mednykh}, hence it is arithmetic \cite{hlm-arithmeticity} with fundamental group the Fibonacci group
\[ F(2,8) = \langle x_0,x_1,\dots,x_7 \mid x_i x_{i+1} \equiv x_{i+2} \ \forall i \in \Z/8\Z \rangle. \]
(It would be interesting to prove directly from this presentation that all representations $F(2,8) \to SU(2)$ have cyclic image.)  In the Hodgson--Weeks census of closed hyperbolic 3-manifolds it is identified as $m036(-3,2)$, and it cannot be Dehn surgery on a knot in $S^3$ since $H_1(\dcover(8_{18})) = \Z/3\Z \oplus \Z/15\Z$ is not cyclic.  However, the one-cusped hyperbolic manifold $m036$ has at least four $SU(2)$-cyclic fillings, including
\begin{align*}
m036(1,0) &= L(3,1) & 
m036(-1,1) &= L(21,8) \\
m036(0,1) &= S^2((3,1),(3,2),(3,-1)) &
m036(-3,2) &= \dcover(8_{18})
\end{align*}
according to \cite{dunfield-exceptional}.  Notably, the $\dcover(8_{18})$ slope has distance three from the Seifert fibered, non-lens space slope.
\end{example}

\begin{example} \label{ex:m100}
The 3-bridge knot $10_{109}$ is also $SU(2)$-simple \cite[\S 6]{cornwell}, so $\dcover(10_{109})$ is $SU(2)$-cyclic.  It is identified in the hyperbolic census as $m100(2,3)$.  We remark that $m100$ also has two lens space fillings, namely
\begin{align*}
m100(1,0) &= \RP^3, & m100(1,1) &= L(29,12)
\end{align*}
according to \cite{dunfield-exceptional}.  Here the $\dcover(10_{109})$ slope is distance three from the $\RP^3$ slope.
\end{example}

\subsection{Obstructions to 1-domination} \label{ssec:obstruction-1-domination}

By a result of Liu--Sun \cite{Liu-Sun}, given any 3-manifold $N$ and a hyperbolic 3-manifold $Y$, there is some finite cover $\tilde{Y}$ of $Y$ which admits a degree-1 map $\tilde{Y} \to N$.  Such a map is also called a $1$-domination. In short, every hyperbolic 3-manifold virtually 1-dominates any other 3-manifold.  The following shows that the adverb ``virtually'' cannot be removed.

\begin{theorem} \label{thm:obstruction-1-domination}
Let $Y$ be an $SU(2)$-cyclic hyperbolic rational homology 3-sphere, with $H_1(Y;\Z)$ of odd order.  Then $Y$ does not admit a map of degree 1 to any Seifert fibered 3-manifold with three or more singular fibers.
\end{theorem}

Examples of such $Y$ include the branched double covers $\dcover(8_{18})$ and $\dcover(10_{109})$.

\begin{proof}[Proof of Theorem~\ref{thm:obstruction-1-domination}]
A 1-domination $Y \to Y'$ induces an epimorphism of fundamental groups, and hence also an epimorphism $H_1(Y;\Z) \to H_1(Y';\Z)$.   If some representation $\rho: \pi_1(Y') \to SU(2)$ has non-cyclic image then so does the composition
\[ \pi_1(Y) \to \pi_1(Y') \xrightarrow{\rho} SU(2), \]
contradicting the assumption that $Y$ is $SU(2)$-cyclic, and so $Y'$ must also be $SU(2)$-cyclic.

We now look to Theorem~\ref{thm:su2-abelian-sfs}, which classifies the $SU(2)$-cyclic Seifert fibered spaces $Y'$.  It implies that if $Y'$ is a rational homology sphere with at least three singular fibers (i.e.,\ not $S^3$ or a lens space) then the group $H_1(Y';\Z)$ must have even order.  Since $H_1(Y;\Z)$ has odd order it cannot surject onto $H_1(Y';\Z)$, and so $Y$ cannot 1-dominate $Y'$ after all.
\end{proof}

\begin{remark}
There are clearly similar statements one can make about rational homology spheres of even order or of manifolds with positive first Betti number, but we do not have any examples to which these would apply.
\end{remark}

\section{Cyclic covers and $SU(2)$} \label{sec:cyclic-covers}

Following Boyer and Nicas \cite{boyer-nicas}, we say that $\pi_1(Y)$ is \emph{cyclically finite} if each regular cyclic cover of $Y$, other than those of maximal even degree, is a rational homology sphere.  These are precisely the covers with fundamental groups of the form $\ker(\mathrm{ad}(\rho))$, where $\rho: \pi_1(Y) \to SU(2)$ is a representation with cyclic image.

Baldwin and the first author \cite[Theorem~4.6]{bs-stein} proved that an $SU(2)$-cyclic rational homology sphere $Y$ is an L-space in the sense of framed instanton homology, provided that $\pi_1(Y)$ is cyclically finite, but this hypothesis was satisfied in all known examples.  Indeed, the Chern-Simons functional defining $I^\#(Y)$ has critical set $\Hom(\pi_1(Y),SU(2))$, whose homology is $\ZZ^{|H_1(Y)|}$ if $Y$ is $SU(2)$-cyclic, and $\pi_1(Y)$ is cyclically finite if and only if this functional is Morse--Bott at each reducible representation.  In this case a Morse--Bott spectral sequence then shows that $I^\#(Y)$ has rank at most $|H_1(Y)|$, which makes $Y$ an L-space.

\begin{proposition} \label{prop:non-cyclically-finite}
Let $Y$ be a Seifert fibered rational homology sphere with base orbifold $S^2(3,3,3)$ and $|H_1(Y;\Z)|$ even.  Then $Y$ is $SU(2)$-cyclic, but $\pi_1(Y)$ is not cyclically finite.
\end{proposition}

\begin{proof}
Writing $Y = S^2((3,\beta_1),(3,\beta_2),(3,\beta_3))$, we note that $Y$ is $SU(2)$-cyclic by Theorem~\ref{thm:small-seifert-fibered}.  We will construct a normal, 3-fold cyclic cover $p: \tilde{Y} \to Y$ which is a circle bundle over $T^2$ and thus satisfies $b_1(\tilde{Y}) \geq b_1(T^2) = 2$.

The key observation is that there is a regular, 3-fold cover $T^2 \to S^2(3,3,3)$.  To construct it, we realize the torus as $T^2 = \C/\Lambda$, where $\Lambda \subset \C$ is the lattice spanned by $1$ and $\omega = e^{2\pi i/3}$.  Then multiplication by $\omega$ fixes $\Lambda$ and hence descends to a $\Z/3\Z$-action on $T^2$ with fixed points $0, \frac{1}{3}(1+2\omega)$, and $\frac{1}{3}(2+\omega)$; the quotient by this action gives the covering $\pi: T^2 \to S^2(3,3,3)$.  We pull back the circle bundle $Y \to S^2(3,3,3)$ to $T^2$ along $\pi$ to get a circle bundle $\tilde{Y} \to T^2$ which is a regular, 3-fold cover of $Y$, as desired.
\end{proof}

\begin{remark}
The Seifert fibration $Y \to S^2(3,3,3)$ in Proposition~\ref{prop:non-cyclically-finite} has Euler number
\[ e(Y) = -\sum_{i=1}^3 \frac{\beta_i}{\alpha_i} = -\frac{1}{3} \sum_{i=1}^3 \beta_i, \]
and $3e(Y) = -\sum \beta_i$ is even and nonzero since Lemma~\ref{lem:homology-sfs} tells us that $|H_1(Y)| = 27|e(Y)|$. 
The Euler number of the circle bundle $\tilde{Y} \to T^2$ is $e(\tilde{Y}) = 3e(Y)$ by \cite[Theorem~3.3]{jankins-neumann}, and since this is a nonzero, even integer, the cover $\tilde{Y}$ is also $SU(2)$-abelian by Theorem~\ref{thm:su2-abelian-sfs}.
\end{remark}

It is also natural to ask whether being $SU(2)$-abelian is preserved under taking covers, and likewise for not being $SU(2)$-abelian.  In both cases, the answer is no.

\begin{proposition}
There are infinitely many double covers of the form $\tilde{Y} \to Y$, with $\tilde{Y}$ and $Y$ both rational homology spheres, for which $\tilde{Y}$ is $SU(2)$-cyclic but $Y$ is not.
\end{proposition}

\begin{proof}
We take $Y$ to be any prism manifold and $\tilde{Y} \to Y$ a double cover by a lens space.  In general, prism manifolds are Seifert fibered over $S^2(2,2,n)$ for some $n$, and the lens space covering comes from pulling back this fibration along the double cover $S^2(n,n) \to S^2(2,2,n)$.  The latter can be easily constructed from the double cover $S^2 \to S^2(2,2)$, corresponding to rotation of $S^2$ by $\pi$ around an axis through the north and south poles, by adding one orbifold point of order $n$ to the base and two such points at its preimages in the cover.
\end{proof}

\begin{proposition}
There are infinitely many double covers of the form $\tilde{Y} \to Y$, with $\tilde{Y}$ and $Y$ both rational homology spheres, for which $Y$ is $SU(2)$-cyclic but $\tilde{Y}$ is not.
\end{proposition}

\begin{proof}
We identify $T^2 = (\R/\Z)^2$ and define a diffeomorphism
\[ f:T^2 \to T^2, \qquad f(x,y) = (y,1-x) \]
of order 4.  Then $f$ has fixed points $(0,0)$ and $(\frac{1}{2},\frac{1}{2})$, and a single orbit
\[ \{ (\tfrac{1}{2},0), (0,\tfrac{1}{2}) \} \]
of order 2, and all other points of $T^2$ lie in orbits of order $4$. Moreover, the involution $f^2$ has four fixed points, namely
\[ (0,0), (0,\tfrac{1}{2}), (\tfrac{1}{2},0), (\tfrac{1}{2},\tfrac{1}{2}), \]
It follows that the quotient $T^2 / f^2$ is the orbifold $S^2(2,2,2,2)$, while $T^2 / f = S^2(2,4,4)$, and so we have a double cover
\[ S^2(2,2,2,2) \to S^2(2,4,4). \]
We now take any Seifert fibered rational homology sphere $Y \to S^2(2,4,4)$ and pull the fibration back along this double cover to get a double cover $\tilde{Y} \to Y$.  Since $\tilde{Y}$ is Seifert fibered with base $S^2(2,2,2,2)$, however, Theorem~\ref{thm:large-seifert-fibered} says that it is not $SU(2)$-abelian.

To see that $\tilde{Y}$ is a rational homology sphere if $Y$ is, a slight generalization of Lemma~\ref{lem:homology-sfs} says that a Seifert fibered space
\[ S^2((\alpha_1,\beta_1),\dots,(\alpha_n,\beta_n)), \]
with Euler number $e = -\sum \frac{\beta_i}{\alpha_i}$, has first homology of order $\alpha_1\dots\alpha_n \cdot |e|$.  Thus $H_1(Y)$ is finite if and only if $e(Y)$ is nonzero; then $e(\tilde{Y}) = 2e(Y) \neq 0$ by \cite[Theorem~3.3]{jankins-neumann}, and so $H_1(\tilde{Y})$ is finite of order $16|e(\tilde{Y})| = 32|e(Y)| = |H_1(Y)|$.
\end{proof}

\section{Hyperbolic manifolds with many $SU(2)$-cyclic Dehn fillings} \label{sec:four-fillings}

In \cite{sivek-zentner}, we conjectured that the only knots in $S^3$ with infinitely many $SU(2)$-cyclic surgeries are torus knots.  We are still unable to prove this conjecture, or even to find a hyperbolic knot in $S^3$ with more than three non-trivial $SU(2)$-cyclic surgeries.  

In this section we look to other manifolds with torus boundary, and we construct infinitely many one-cusped hyperbolic manifolds which each have at least four $SU(2)$-cyclic Dehn fillings.  
In Remark~\ref{rem:infinitely-many-example} we will also construct a manifold $Y$ with non-trivial JSJ decomposition such that $Y$ has infinitely many $SU(2)$-cyclic Dehn fillings.  We wish to draw attention to Dunfield's census of exceptional Dehn fillings \cite{dunfield-exceptional}, which along with SnapPy \cite{snappy} was instrumental in finding these examples.

In our examples, we use a family of $SU(2)$-cyclic graph manifolds first studied by Motegi \cite{motegi}.  For notation, we let $E_{p,q}$ denote the exterior of any nontrivial torus knot $T_{p,q}$, with meridian $\mu_{p,q}$ and longitude $\lambda_{p,q}$ in $\partial E_{p,q}$.  Then $E_{p,q}$ is Seifert fibered over the disk, and $\sigma_{p,q} = (\mu_{p,q})^{pq}\lambda_{p,q}$ is a generic fiber.

\begin{definition} \label{def:splicing}
Let $T_{a,b}$ and $T_{c,d}$ be two nontrivial torus knots.  The graph manifold
\[ Y(T_{a,b},T_{c,d}) = E_{a,b} \cup_{T^2} E_{c,d} \]
is formed by gluing their exteriors by an orientation-reversing map $\partial E_{a,b} \to \partial E_{c,d}$ which sends $\mu_{a,b}$ to $\sigma_{c,d}$ and $\sigma_{a,b}$ to $\mu_{c,d}$.
\end{definition}

\begin{proposition}[\cite{motegi, zentner-simple, ni-zhang}]
Every $Y(T_{a,b},T_{c,d})$ is $SU(2)$-cyclic.
\end{proposition}

\begin{proposition} \label{prop:splicing-facts}
Let $Y = Y(T_{a,b},T_{c,d})$.  Then $H_1(Y)$ is cyclic of order $|abcd-1|$, and $Y$ is not Seifert fibered.  The torus $\partial E_{a,b}$ is the unique closed, incompressible surface in $Y$ up to isotopy.
\end{proposition}

\begin{proof}
The claim about $H_1(Y)$ is \cite[Lemma~2]{motegi}.  To see that $Y$ is not Seifert fibered, we note that among the $SU(2)$-cyclic Seifert fibered spaces listed in Theorem~\ref{thm:su2-abelian-sfs}, the only ones with finite cyclic first homology are lens spaces.  On the other hand, $\pi_1(Y)$ is infinite since $Y$ contains an incompressible torus, so $Y$ cannot be a lens space.

Suppose that $S \subset Y$ is a closed, incompressible surface which is not isotopic to $\partial E_{a,b}$; we arrange for it to meet $\partial E_{a,b}$ transversely and in as few components as possible.  The components of $S \cap E_{a,b}$ are incompressible and boundary-incompressible, so they must be either annuli or Seifert surfaces, with boundary slope equal to that of $\sigma_{a,b}$ or $\lambda_{a,b}$ respectively \cite{tsau}.  In either case the boundary slope is distance 1 from that of $\mu_{a,b}$.  But $\partial (S \cap E_{c,d})$ likewise consists of curves parallel to either $\lambda_{c,d}$ or $\sigma_{c,d}$, whose distances to $\sigma_{c,d} = \mu_{a,b}$ are $|cd| > 1$ and $0$ respectively, and so these cannot be parallel to any of the components of $\partial (S \cap E_{a,b})$.  We conclude that $S \cap \partial E_{a,b}$ must be empty, but then $S$ is contained in one of $E_{a,b}$ and $E_{c,d}$, and this is only possible if $S$ is parallel to the boundary torus.
\end{proof}

We now begin our discussion of manifolds with many $SU(2)$-cyclic fillings.   We start with the only two hyperbolic knots we currently know of in $S^3$ with three nontrivial $SU(2)$-cyclic Dehn surgeries.
\begin{example} \label{ex:22n0-surgeries}
The one-cusped hyperbolic manifold $m016$ is the complement of the pretzel knot $P = P(-2,3,7) = 12n_{242}$.  In addition to its $S^3$ filling, it has two lens space fillings
\[ S^3_{18}(P) = L(18,5) \quad\mathrm{and}\quad S^3_{19}(P) = L(19,7), \]
and the surgery of slope $\frac{37}{2}$ is the toroidal, $SU(2)$-cyclic manifold $Y(T_{2,3},T_{-2,3})$.  Even without this identification, SnapPy says that
\[ \pi_1(S^3_{37/2}(P)) = \langle a,b \mid (a^3b)^2=b^3,\ (a^{-1}b^3)^2=a^{-3} \rangle, \]
and any $SU(2)$-representation $\rho$ is easily seen to have cyclic image:  if $\rho(a)^3=\pm1$ then the first relator gives $\rho(b)=1$, forcing the image to be cyclic, and likewise if $\rho(b)^3=
\pm1$ using the second relator.  Otherwise $\rho(b)^3 = \rho(a^3b)^2$ lies in a unique $U(1)$ subgroup of $SU(2)$, which then contains both $\rho(b)$ and $\rho(a^3b)$ and hence also $\rho(a^3)$.  But $\rho(a)^3\neq 1$, so $\rho(a)$ must lie in this subgroup as well and then $\rho$ has abelian (hence cyclic) image.

Similarly, the manifold $m118$ is the complement of the knot labeled $k4_4$ in the Callahan--Dean--Weeks census of hyperbolic knots \cite{callahan-dean-weeks}, where it is also identified as a twisted torus knot $T(4,7)_{2,1}$, i.e., the closure of the 4-stranded braid $(\sigma_3\sigma_2\sigma_1)^7 \sigma_1^2$.  It has lens space surgeries
\[ S^3_{30}(k4_4) = L(30,11) \quad\mathrm{and}\quad S^3_{31}(k4_4) = L(31,12), \]
and its $\frac{61}{2}$-surgery is the toroidal, $SU(2)$-cyclic manifold $Y(T_{-2,3},T_{2,5})$ up to orientation.  Indeed, SnapPy gives a presentation
\[ \pi_1(S^3_{61/2}(k4_4)) = \langle a,b \mid (a^5b)^2=b^3,\ (a^{-2}b^3)^2=a^{-5} \rangle, \]
and again one can directly prove this to be $SU(2)$-cyclic: one first shows that $\rho(a)^5=\pm1$ or $\rho(b)^3=\pm1$ implies $\rho(b)=1$ or $\rho(a)=1$, and if neither of these holds then $\rho(b)^3=\rho(a^5b)^2$ lies in a unique $U(1)$ subgroup, which contains $\rho(b)$ and $\rho(a^5b)$ and hence $\rho(a^5)$, and then also $\rho(a)$ since $\rho(a^5) \neq \pm 1$, so $\rho(a)$ and $\rho(b)$ commute after all.
\end{example}

We now give a family of infinitely many manifolds with four $SU(2)$-cyclic fillings each.  The starting point for our construction is \cite[Theorem~6.5]{zentner-simple}, which describes each $Y=Y(T_{a,b},T_{c,d})$ as the branched double cover of an alternating link $L(T_{a,b},T_{c,d})$; we will replace a crossing of this link (actually a knot for the cases in question) with any of several rational tangles, and then the Montesinos trick tells us that the resulting branched double covers all arise as Dehn surgeries on a single knot in $Y$.

\begin{theorem} \label{thm:four-fillings}
For each integer $g \geq 1$ there is a manifold $M_g$ with torus boundary which admits at least four $SU(2)$-cyclic fillings.  Three of these are the lens spaces
\[ L(2g+5,2),\ L(11g+3,11),\ \mathrm{and}\ L(13g+8,13), \]
and one is the graph manifold $Y(T_{2,3},T_{2,2g+1})$.
\end{theorem}

\begin{proof}
We consider the knot $L_g = L(T_{2,3},T_{2,2g+1})$ from \cite[\S 6]{zentner-simple} whose branched double cover is $Y_g = Y(T_{2,3},T_{2,2g+1})$.  It is illustrated in Figure~\ref{fig:knot-Lg}.
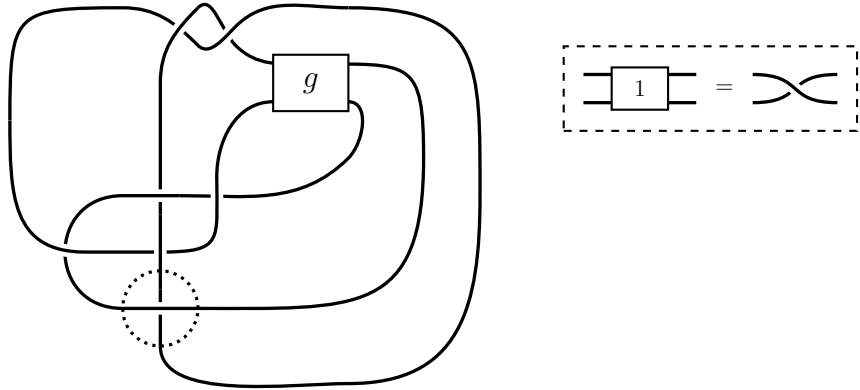
\begin{figure}
\begin{tikzpicture}
\tikzset{twistregion/.style={draw, fill=white, thick, minimum width=1cm, minimum height=0.75cm}}
\draw[link] (0,-2) to[out=90, in=270] (0,1.5) to[out=90,in=225,looseness=1] (0.5,2.5) to [out=45,in=135,looseness=1] (1,2) to[out=315,in=180,looseness=1] (2,1.75);
\draw[link] (2,1.75) to[out=0,in=90] (3.5,0.5) to [out=270,in=0,looseness=1.5] (0.5,-1.5) to[out=180,in=0] (-0.5,-1.5) to[out=180,in=180] (-0.5,0) to[out=0,in=180] (0.25,0);
\draw[link] (0.25,0) to[out=0, in=225,looseness=1] (2.5,0.5) to[out=45,in=0,looseness=1] (2.5,1.25) to[out=180,in=0] (1.5,1.25);
\draw[link] (1.5,1.25) to[out=180, in=90,looseness=1] (0.75,0.25) to[out=270,in=0] (0,-0.75) to[out=180,in=0] (-1,-0.75) to [out=180,in=270,looseness=1.25] (-2,1);
\draw[link] (-2,1) to[out=90,in=180] (-0.5,2.5) to[out=0,in=135,looseness=1] (0.5,2) to[out=315,in=225,looseness=1] (1,2.25) to[out=45,in=180,looseness=1] (2.5,2.5);
\draw[link] (2.5,2.5) to[out=0,in=90,looseness=1.5] (4.25,0) to[out=270,in=0,looseness=1.25] (2.5,-2.5) to[out=180,in=270,looseness=0.75] (0,-2);
\node[twistregion] at (2,1.5) {$g$};
\draw[link] (0,-1.25) to[out=90,in=270] (0,-0.25);
\draw[link] (0,1.5) to[out=90,in=225,looseness=1] (0.5,2.5);
\draw[very thick,dotted] (0,-1.5) circle (0.5);

\begin{scope}[scale=0.75, xshift=-0.5cm, yshift=1.9cm, transform shape]
\draw [dashed] (7.65,0.75) rectangle (12.85, -0.75);
\draw[link] (8,0.25) -- (10,0.25) (8,-0.25) -- (10,-0.25);
\node [twistregion] at (9,0) {$1$};
\node at (10.5,0) {$=$};
\draw[link] (11,-0.25) to[out=0,in=180,looseness=1.5] (12.5,0.25);
\draw[link] (11,0.25) to[out=0,in=180,looseness=1.5] (12.5,-0.25);
\end{scope}

\end{tikzpicture}
\caption{The knot $L_g = L(T_{2,3},T_{2,2g+1})$.} \label{fig:knot-Lg}
\end{figure}

We claim that the circled crossing in Figure~\ref{fig:knot-Lg} can be replaced with any of three rational tangles
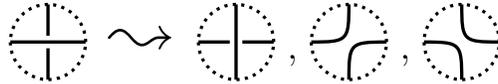
\[ \begin{tikzpicture}
\draw[link] (0,-0.5) -- (0,0.5);
\draw[link] (-0.5,0) -- (0.5,0);
\draw[very thick, dotted] (0,0) circle (0.5);
\node at (1.25,0) {\Huge $\leadsto$};
\draw[link] (2,0) -- (3,0);
\draw[link] (2.5,-0.5) -- (2.5,0.5);
\draw[very thick, dotted] (2.5,0) circle (0.5);
\node at (3.25,-0.25) {\large ,};
\draw[link] (3.5,0) to[out=0,in=270] (4,0.5);
\draw[link] (4,-0.5) to[out=90,in=180] (4.5,0);
\draw[very thick, dotted] (4,0) circle (0.5);
\node at (4.75,-0.25) {\large ,};
\draw[link] (5,0) to[out=0,in=90] (5.5,-0.5);
\draw[link] (5.5,0.5) to[out=270,in=180] (6,0);
\draw[very thick, dotted] (5.5,0) circle (0.5);
\end{tikzpicture} \]
such that each tangle replacement converts $L_g$ into a 2-bridge link.  Each of these corresponds to a Dehn surgery on the same knot $K_g$ in $Y_g = \Sigma_2(L_g)$, a neighborhood of which is the preimage of a small ball containing the original crossing.  The Dehn surgery turns $Y_g$ into the branched double cover of the corresponding 2-bridge link, which is a lens space.  These tangle replacements are carried out in Figures~\ref{fig:Lg-1}, \ref{fig:Lg-2}, and \ref{fig:Lg-3}, along with some isotopies to make it visible that they are indeed 2-bridge links.

By putting the first of these 2-bridge links in normal form, we see that it corresponds to the continued fraction
\[ [g+2,2] = (g+2) + \frac{1}{2} = \frac{2g+5}{2}; \]
similarly, the second and third have continued fractions
\[ [g+1,-2,1,2,-2,-1] = \frac{11g+3}{11} \quad\mathrm{and}\quad [g,1,1,1,1,2] = \frac{13g+8}{13}. \]
Their branched double covers are thus the lens spaces $L(2g+5,2)$, $L(11g+3,11)$, and $L(13g+8,13)$ respectively.
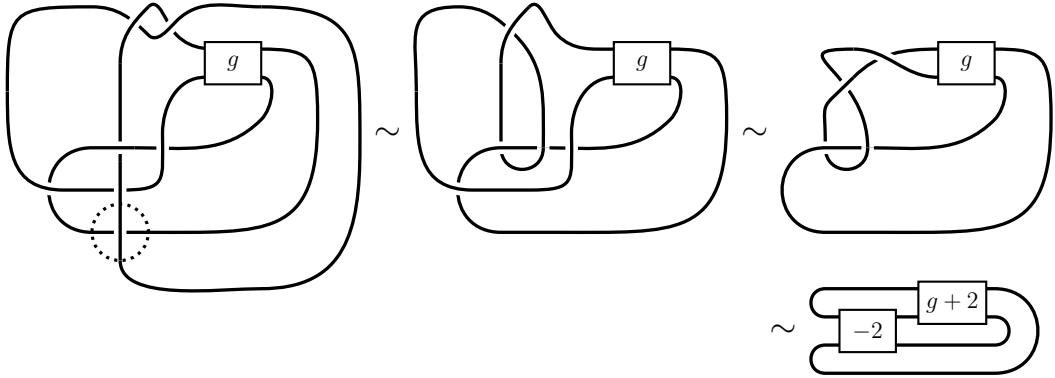
\begin{figure}
\begin{tikzpicture}[scale=0.75, transform shape]
\tikzset{twistregion/.style={draw, fill=white, thick, minimum width=1cm, minimum height=0.75cm}}
\begin{scope}
\draw[link] (0,-2) to[out=90, in=270] (0,1.5) to[out=90,in=225,looseness=1] (0.5,2.5) to [out=45,in=135,looseness=1] (1,2) to[out=315,in=180,looseness=1] (2,1.75);
\draw[link] (2,1.75) to[out=0,in=90] (3.5,0.5) to [out=270,in=0,looseness=1.5] (0.5,-1.5) to[out=180,in=0] (-0.5,-1.5) to[out=180,in=180] (-0.5,0) to[out=0,in=180] (0.25,0);
\draw[link] (0.25,0) to[out=0, in=225,looseness=1] (2.5,0.5) to[out=45,in=0,looseness=1] (2.5,1.25) to[out=180,in=0] (1.5,1.25);
\draw[link] (1.5,1.25) to[out=180, in=90,looseness=1] (0.75,0.25) to[out=270,in=0] (0,-0.75) to[out=180,in=0] (-1,-0.75) to [out=180,in=270,looseness=1.25] (-2,1);
\draw[link] (-2,1) to[out=90,in=180] (-0.5,2.5) to[out=0,in=135,looseness=1] (0.5,2) to[out=315,in=225,looseness=1] (1,2.25) to[out=45,in=180,looseness=1] (2.5,2.5);
\draw[link] (2.5,2.5) to[out=0,in=90,looseness=1.5] (4.25,0) to[out=270,in=0,looseness=1.25] (2.5,-2.5) to[out=180,in=270,looseness=0.75] (0,-2);
\node[twistregion] at (2,1.5) {$g$};
\draw[link] (0,-1.25) to[out=90,in=270] (0,-0.25);
\draw[link] (0,1.5) to[out=90,in=225,looseness=1] (0.5,2.5);
\draw[link] (0,-2) to[out=90,in=270] (0,-1);
\draw[very thick,dotted] (0,-1.5) circle (0.5);
\end{scope}
\node at (4.74,0.25) {\Large $\sim$};
\begin{scope}[xshift=7.25cm]
\draw[link] (-0.5,1.5) to[out=90,in=225,looseness=1] (0,2.5) to [out=45,in=135,looseness=1] (0.5,2) to[out=315,in=180,looseness=1] (1.5,1.75);
\draw[link] (2,1.75) to[out=0,in=90] (3.5,0.5) to [out=270,in=0,looseness=1.5] (0.5,-1.5) to[out=180,in=0] (-0.5,-1.5) to[out=180,in=180] (-0.5,0) to[out=0,in=180] (0.25,0);
\draw[link] (0.25,0) to[out=0, in=225,looseness=1] (2.5,0.5) to[out=45,in=0,looseness=1] (2.5,1.25) to[out=180,in=0] (1.5,1.25);
\draw[link] (1.5,1.25) to[out=180, in=90,looseness=1] (0.75,0.25) to[out=270,in=0] (0,-0.75) to[out=180,in=0] (-1,-0.75) to [out=180,in=270,looseness=1.25] (-2,1);
\draw[link] (-2,1) to[out=90,in=180,looseness=1.5] (-1.25,2.5) to[out=0,in=135,looseness=1] (-0.25,2) to[out=315,in=90,looseness=1] (0.25,0);
\draw[link] (0.25,0) to[out=270,in=270,looseness=1.75] (-0.5,0) to[out=90,in=270] (-0.5,1.5);
\node[twistregion] at (2,1.5) {$g$};
\draw[link] (-0.5,1.5) to[out=90,in=225,looseness=1] (0,2.5);
\begin{scope} \clip (-0.75,-0.25) rectangle (-0.25,0.25);
\draw[link] (-0.5,-1.5) to[out=180,in=180] (-0.5,0) to[out=0,in=180] (0.25,0);
\end{scope}
\end{scope}
\node at (11.25,0.25) {\Large $\sim$};
\begin{scope}[xshift=13cm]
\draw[link] (1.5,1.75) to[out=180,in=45,looseness=1] (0,1.25);
(-0.5,1.5) to[out=90,in=225,looseness=1] (0,2.5) to [out=45,in=135,looseness=1] (0.5,2) to[out=315,in=180,looseness=1] (1.5,1.75);
\draw[link] (2,1.75) to[out=0,in=90] (3.5,0.5) to [out=270,in=0,looseness=1.5] (0.5,-1.5) to[out=180,in=0] (-0.5,-1.5) to[out=180,in=180] (-0.5,0) to[out=0,in=180] (0.25,0);
\draw[link] (0.25,0) to[out=0, in=225,looseness=1] (2.5,0.5) to[out=45,in=0,looseness=1] (2.5,1.25) to[out=180,in=0] (1.5,1.25);
\draw[link] (1.5,1.25) to[out=180,in=0,looseness=1] (0,1.75);
\draw[link] (0.25,0) to[out=90,in=180,looseness=2] (0,1.75);
\draw[link] (0.25,0) to[out=270,in=270,looseness=1.75] (-0.5,0) to[out=90,in=225] (0,1.25);
\node[twistregion] at (2,1.5) {$g$};
\begin{scope} \clip (-0.75,-0.25) rectangle (-0.25,0.25);
\draw[link] (-0.5,-1.5) to[out=180,in=180] (-0.5,0) to[out=0,in=180] (0.25,0);
\end{scope}
\end{scope}

\node at (11.75,-3.25) {\Large $\sim$};
\begin{scope}[xshift=12.5cm, yshift=-4cm]
\tikzset{twistregion/.style={draw, fill=white, thick, minimum width=1cm, minimum height=0.75cm}}
\draw[link] (0,1.5) to[out=180,in=180] (0,1) to[out=0,in=180] (3,1) to[out=0,in=0] (3,0.5) to[out=180,in=0] (0,0.5) to[out=180,in=180] (0,0) to[out=0,in=180] (3,0) to[out=0,in=0] (3,1.5) -- cycle;
\node[twistregion] at (0.75,0.75) {$-2$};
\node[twistregion] at (2.25,1.25) {$g+2$};
\end{scope}
\end{tikzpicture}
\caption{The first tangle replacement on $L_g$, followed by an isotopy.} \label{fig:Lg-1}
\end{figure}
\begin{figure}
\begin{tikzpicture}[scale=0.75, transform shape]
\tikzset{twistregion/.style={draw, fill=white, thick, minimum width=1cm, minimum height=0.75cm}}
\clip (-2.25,2.8) rectangle (18.25,-4.3);
\begin{scope}
\draw[link] (0,-2) to[out=90, in=270] (0,1.5) to[out=90,in=225,looseness=1] (0.5,2.5) to [out=45,in=135,looseness=1] (1,2) to[out=315,in=180,looseness=1] (2,1.75);
\draw[link] (2,1.75) to[out=0,in=90] (3.5,0.5) to [out=270,in=0,looseness=1.5] (0.5,-1.5) to[out=180,in=0] (-0.5,-1.5) to[out=180,in=180] (-0.5,0) to[out=0,in=180] (0.25,0);
\draw[link] (0.25,0) to[out=0, in=225,looseness=1] (2.5,0.5) to[out=45,in=0,looseness=1] (2.5,1.25) to[out=180,in=0] (1.5,1.25);
\draw[link] (1.5,1.25) to[out=180, in=90,looseness=1] (0.75,0.25) to[out=270,in=0] (0,-0.75) to[out=180,in=0] (-1,-0.75) to [out=180,in=270,looseness=1.25] (-2,1);
\draw[link] (-2,1) to[out=90,in=180] (-0.5,2.5) to[out=0,in=135,looseness=1] (0.5,2) to[out=315,in=225,looseness=1] (1,2.25) to[out=45,in=180,looseness=1] (2.5,2.5);
\draw[link] (2.5,2.5) to[out=0,in=90,looseness=1.5] (4.25,0) to[out=270,in=0,looseness=1.25] (2.5,-2.5) to[out=180,in=270,looseness=0.75] (0,-2);
\node[twistregion] at (2,1.5) {$g$};
\draw[link] (0,-1.25) to[out=90,in=270] (0,-0.25);
\draw[link] (0,1.5) to[out=90,in=225,looseness=1] (0.5,2.5);
\draw[link] (0,-2) to[out=90,in=270] (0,-1);
\draw[very thick,dotted,fill=white] (0,-1.5) circle (0.5);
\draw[link] (-0.5,-1.5) to[out=0,in=270] (0,-1);
\draw[link] (0,-2) to[out=90,in=180] (0.5,-1.5);
\draw[very thick,dotted] (0,-1.5) circle (0.5);
\end{scope}
\node at (4.74,0.25) {\Large $\sim$};
\begin{scope}[xshift=7.25cm]
\draw[link] (0,-1) to[out=90, in=270] (0,1.5) to[out=90,in=225,looseness=1] (0.5,2.5) to [out=45,in=135,looseness=1] (1,2) to[out=315,in=180,looseness=1] (2,1.75);
\begin{scope} \clip (-0.25,-2)  -- (-0.25,-1) -- (0.25,-1) -- (0.25,0.25) -- (-1.5,0.25) -- (-1.5,-2) -- cycle;
\draw[link] (0.5,-1.5) to[out=180,in=0] (-0.5,-1.5) to[out=180,in=180] (-0.5,0) to[out=0,in=180] (0.25,0);
\end{scope}
\draw[link] (2.5,1.25) to[out=0,in=180,looseness=1] (3.5,1.75) to[out=0,in=0] (3.5,0) to[out=180,in=0] (0.25,0);
\draw[link] (1.5,1.25) to[out=180, in=90,looseness=1] (0.75,0.25) to[out=270,in=0] (0,-0.75) to[out=180,in=0] (-1,-0.75) to [out=180,in=270,looseness=1.25] (-2,1);
\draw[link] (2.5,1.75) to[out=0,in=180,looseness=1] (3.5,1.25) to[out=0,in=0] (3.5,0.5) to[out=180,in=0] (2,0.5) to[out=180,in=315,looseness=1] (0.5,2) to[out=135,in=0,looseness=1] (-0.5,2.5) to[out=180,in=90] (-2,1);
\node[twistregion] at (2,1.5) {$g$};
\draw[link] (0,-1.25) to[out=90,in=270] (0,-0.25);
\draw[link] (0,1.5) to[out=90,in=225,looseness=1] (0.5,2.5);
\draw[link] (-0.5,-1.5) to[out=0,in=270] (0,-1);
\end{scope}
\node at (12.25,0.25) {\Large $\sim$};
\begin{scope}[xshift=13.5cm]
\draw[link] (0,-1) to[out=90, in=270] (0,0.75) to[out=90,in=180,looseness=1] (1,1.75) to[out=0,in=180] (2,1.75);
\begin{scope} \clip (-0.25,-2)  -- (-0.25,-1) -- (0.25,-1) -- (0.25,0.25) -- (-1.5,0.25) -- (-1.5,-2) -- cycle;
\draw[link] (0.5,-1.5) to[out=180,in=0] (-0.5,-1.5) to[out=180,in=180] (-0.5,0) to[out=0,in=180] (0.25,0);
\end{scope}
\draw[link] (0,-2) to[out=0,in=315] (0.75,0.75) to[out=135,in=135,looseness=3] (1.1,1.1) to[out=315,in=180,looseness=1.5] (2,0.5) to[out=0,in=180] (3.5,0.5) to[out=0,in=0] (3.5,1.25) to[out=180,in=0,looseness=1] (2.5,1.75);
\draw[link] (2.5,1.25) to[out=0,in=180,looseness=1] (3.5,1.75) to[out=0,in=0] (3.5,0) to[out=180,in=0] (0.25,0);
\draw[link] (1.5,1.25) to[out=180, in=90,looseness=1] (0.75,0.25) to[out=270,in=0] (0,-0.75) to[out=180,in=180] (0,-2);
\draw[link] (0,-1.25) to[out=90,in=270] (0,-0.25);
\draw[link] (3.5,1.25) to[out=180,in=0,looseness=1] (2.5,1.75);
\begin{scope} \clip (0.85,0.85) rectangle (1.35,1.35);
\draw[link] (0.75,0.75) to[out=135,in=135,looseness=3] (1.1,1.1) to[out=315,in=180,looseness=1.5] (2,0.5);
\end{scope}
\node[twistregion] at (2,1.5) {$g$};
\draw[link] (-0.5,-1.5) to[out=0,in=270] (0,-1);
\end{scope}
\node at (7,-3.25) {\Large $\sim$};
\begin{scope}[xshift=7.75cm,yshift=-4cm]
\tikzset{twistregion/.style={draw, fill=white, thick, minimum width=1cm, minimum height=0.75cm}}
\draw[link] (0,1.5) to[out=180,in=180] (0,1) to[out=0,in=180] (9,1) to[out=0,in=0] (9,0.5) to[out=180,in=0] (0,0.5) to[out=180,in=180] (0,0) to[out=0,in=180] (9,0) to[out=0,in=0] (9,1.5) -- cycle;
\node[twistregion] at (0.75,0.75) {$1$};
\node[twistregion] at (2.25,0.25) {$-2$};
\node[twistregion] at (3.75,0.75) {$-2$};
\node[twistregion] at (5.25,0.25) {$1$};
\node[twistregion] at (6.75,0.75) {$2$};
\node[twistregion] at (8.25,1.25) {$g+1$};
\end{scope}
\end{tikzpicture}
\caption{The second tangle replacement on $L_g$, followed by an isotopy.} \label{fig:Lg-2}
\end{figure}
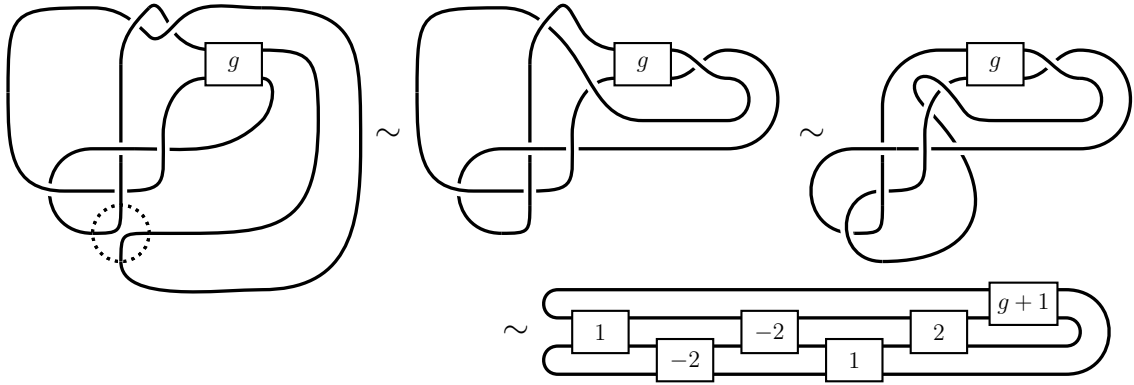
\begin{figure}
\begin{tikzpicture}[scale=0.75, transform shape]
\tikzset{twistregion/.style={draw, fill=white, thick, minimum width=1cm, minimum height=0.75cm}}
\begin{scope}
\draw[link] (0,-2) to[out=90, in=270] (0,1.5) to[out=90,in=225,looseness=1] (0.5,2.5) to [out=45,in=135,looseness=1] (1,2) to[out=315,in=180,looseness=1] (2,1.75);
\draw[link] (2,1.75) to[out=0,in=90] (3.5,0.5) to [out=270,in=0,looseness=1.5] (0.5,-1.5) to[out=180,in=0] (-0.5,-1.5) to[out=180,in=180] (-0.5,0) to[out=0,in=180] (0.25,0);
\draw[link] (0.25,0) to[out=0, in=225,looseness=1] (2.5,0.5) to[out=45,in=0,looseness=1] (2.5,1.25) to[out=180,in=0] (1.5,1.25);
\draw[link] (1.5,1.25) to[out=180, in=90,looseness=1] (0.75,0.25) to[out=270,in=0] (0,-0.75) to[out=180,in=0] (-1,-0.75) to [out=180,in=270,looseness=1.25] (-2,1);
\draw[link] (-2,1) to[out=90,in=180] (-0.5,2.5) to[out=0,in=135,looseness=1] (0.5,2) to[out=315,in=225,looseness=1] (1,2.25) to[out=45,in=180,looseness=1] (2.5,2.5);
\draw[link] (2.5,2.5) to[out=0,in=90,looseness=1.5] (4.25,0) to[out=270,in=0,looseness=1.25] (2.5,-2.5) to[out=180,in=270,looseness=0.75] (0,-2);
\node[twistregion] at (2,1.5) {$g$};
\draw[link] (0,-1.25) to[out=90,in=270] (0,-0.25);
\draw[link] (0,1.5) to[out=90,in=225,looseness=1] (0.5,2.5);
\draw[link] (0,-2) to[out=90,in=270] (0,-1);
\draw[very thick,dotted,fill=white] (0,-1.5) circle (0.5);
\draw[link] (-0.5,-1.5) to[out=0,in=90] (0,-2);
\draw[link] (0,-1) to[out=270,in=180] (0.5,-1.5);
\draw[very thick,dotted] (0,-1.5) circle (0.5);
\end{scope}
\node at (4.74,0.25) {\Large $\sim$};
\begin{scope}[xshift=7.5cm]
\draw[link] (0,-1) to[out=90, in=270] (0,1.5) to[out=90,in=225,looseness=1] (0.5,2.5) to [out=45,in=135,looseness=1] (1,2) to[out=315,in=180,looseness=1] (2,1.75);
\draw[link] (2,1.75) to[out=0,in=90] (3.5,0.5) to [out=270,in=0,looseness=1.5] (0.5,-1.5) to[out=180,in=0] (0.5,-1.5) to[out=180,in=270] (0,-1);
\draw[link] (0.25,0) to[out=0, in=225,looseness=1] (2.5,0.5) to[out=45,in=0,looseness=1] (2.5,1.25) to[out=180,in=0] (1.5,1.25);
\draw[link] (1.5,1.25) to[out=180, in=90,looseness=1] (0.75,0.25) to[out=270,in=0] (0,-0.75) to[out=180,in=0] (-1,-0.75) to [out=180,in=270,looseness=1.25] (-2,1);
\draw[link] (-2,1) to[out=90,in=180] (-0.5,2.5) to[out=0,in=135,looseness=1] (0.5,2) to[out=315,in=225,looseness=1] (1,2.25) to[out=45,in=0,looseness=1] (0.5,3);
\draw[link] (0.5,3) to[out=180,in=90] (-2.5,1) to[out=270,in=180,looseness=1] (-1.5,0) to[out=0,in=180,looseness=1] (-0.5,0) to[out=0,in=180] (0.25,0);
\node[twistregion] at (2,1.5) {$g$};
\draw[link] (0,-1.25) to[out=90,in=270] (0,-0.25);
\draw[link] (0,1.5) to[out=90,in=225,looseness=1] (0.5,2.5);
\draw[link] (-1,-0.75) to [out=180,in=270,looseness=1.25] (-2,1);
\end{scope}
\node at (11.45,0.25) {\Large $\sim$};
\begin{scope}[xshift=14.25cm]
\draw[link] (0,-1) to[out=90, in=270] (0,0.75) to[out=90,in=0,looseness=1] (-1.25,1.25) to[out=180,in=180,looseness=2] (-1.25,1.75) to[out=0,in=180] (1.5,1.75);
\draw[link] (2,1.75) to[out=0,in=90] (3.5,0.5) to [out=270,in=0,looseness=1.5] (0.5,-1.5) to[out=180,in=0] (0.5,-1.5) to[out=180,in=270] (0,-1);
\draw[link] (0.25,0) to[out=0, in=225,looseness=1] (2.5,0.5) to[out=45,in=0,looseness=1] (2.5,1.25) to[out=180,in=0] (1.5,1.25);
\draw[link] (1.5,1.25) to[out=180, in=90,looseness=1] (0.75,0.25) to[out=270,in=0] (0,-0.75) to[out=180,in=0] (0,-0.75) to [out=180,in=270,looseness=1.25] (-1,1);
\draw[link] (-1,1) to[out=90,in=0,looseness=1] (-1.75,2.25) to[out=180,in=90,looseness=1] (-2.5,1) to[out=270,in=180,looseness=1] (-1.5,0) to[out=0,in=180,looseness=1] (-0.5,0) to[out=0,in=180] (0.25,0);
\node[twistregion] at (2,1.5) {$g$};
\draw[link] (0,-1.25) to[out=90,in=270] (0,-0.25);
\begin{scope} \clip (-1.25,-0.5) rectangle (-0.25,0.5);
\draw[link] (0,-0.75) to [out=180,in=270,looseness=1.25] (-1,1);
\end{scope}
\draw[link] (0,0.75) to[out=90,in=0,looseness=1] (-1.25,1.25);
\end{scope}

\node at (7,-3.25) {\Large $\sim$};
\begin{scope}[xshift=7.75cm,yshift=-4cm]
\tikzset{twistregion/.style={draw, fill=white, thick, minimum width=1cm, minimum height=0.75cm}}
\draw[link] (0,1.5) to[out=180,in=180] (0,1) to[out=0,in=180] (9,1) to[out=0,in=0] (9,0.5) to[out=180,in=0] (0,0.5) to[out=180,in=180] (0,0) to[out=0,in=180] (9,0) to[out=0,in=0] (9,1.5) -- cycle;
\node[twistregion] at (0.75,0.75) {$-2$};
\node[twistregion] at (2.25,0.25) {$1$};
\node[twistregion] at (3.75,0.75) {$-1$};
\node[twistregion] at (5.25,0.25) {$1$};
\node[twistregion] at (6.75,0.75) {$-1$};
\node[twistregion] at (8.25,1.25) {$g$};
\end{scope}
\end{tikzpicture}
\caption{The third tangle replacement on $L_g$, followed by an isotopy.} \label{fig:Lg-3}
\end{figure}
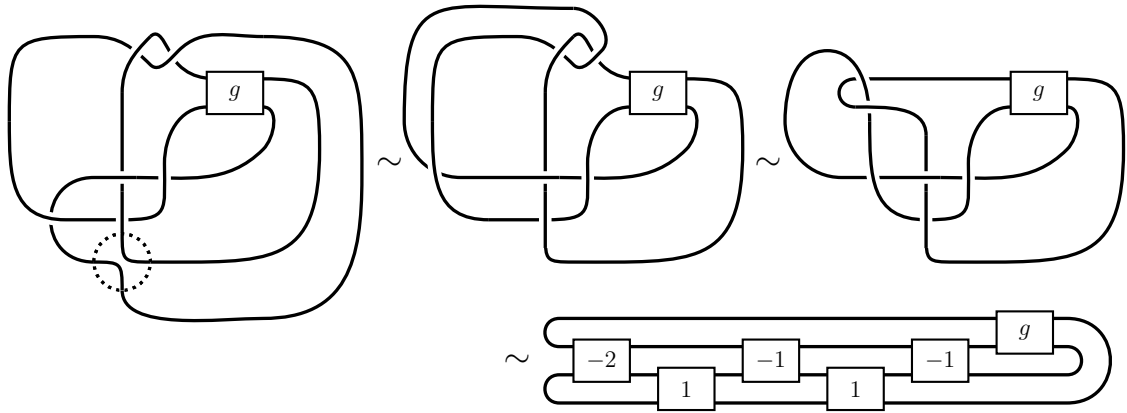
We can therefore take $M_g = Y_g \setminus K_g$.
\end{proof}

In the sequel, we will let $r_T$, $r_2$, $r_{11}$, and $r_{13}$ be the slopes on $\partial M_g$ with Dehn fillings $Y_g$, $L(2g+5,2)$, $L(11g+3,11)$, and $L(13g+8,13)$ respectively.  By construction the distance between $r_T$ and $r_2$ is $2$.

\begin{remark}
The construction of $M_g$ was inspired by looking through \cite{dunfield-exceptional}.  Based on this, we believe, but will not prove, that the first few $M_g$ are the following cusped hyperbolic manifolds:
\[ \def\arraystretch{1.25}
\begin{array}{c|cccccc}
g & 1 & 2 & 3 & 4 & 5 & 6 \\
\hline
M_g & m017 & m119 & m164 & s091 & v0172 & t00333 \\
M_{-g-1} & m016 & m118 & m163 & s092 & v0173 & t00332
\end{array} \]
The manifolds $M_{-g-1}$ in the second row of this table are constructed in the same way as the $M_g$, but with a $(-g-1)$-twist region in place of the $g$-twist region; they have three lens space fillings
\[ -L(2g-3,2), \qquad -L(11g+8,11), \qquad -L(13g+5,13) \]
of slopes $r_2$, $r_{11}$, and $r_{13}$ by exactly the same argument, and the filling $M_{-g-1}(r_T)$ appears to be $\pm Y(T_{2,3},T_{2,-(2g+1)})$.  We note that for $g=1,2$ the lens space filling $M_{-g-1}(r_2)$ is actually $S^3$, and indeed $M_{-2} = m016$ and $M_{-3} = m118$ are the complements in $S^3$ of $P(-2,3,7)$ and the twisted torus knot $k4_4$, which previously appeared in Example~\ref{ex:22n0-surgeries}.
\end{remark}

Most of the remainder of this section will be devoted to proving the hyperbolicity of the $M_g$.

\begin{theorem} \label{thm:M_g-hyperbolic}
The manifolds $M_g$, which have four $SU(2)$-cyclic Dehn fillings by Theorem~\ref{thm:four-fillings}, are hyperbolic for all $g \geq 1$.
\end{theorem}

\begin{proof}
By Thurston's hyperbolization theorem for Haken manifolds \cite{thurston} it suffices to show that $M_g$ has no essential spheres, disks, annuli, or tori.  We will prove the following:
\begin{itemize}
\item $M_g$ is irreducible (Lemma~\ref{lem:mg-irreducible});
\item $M_g$ is not Seifert fibered (also Lemma~\ref{lem:mg-irreducible});
\item All incompressible tori in $M_g$ are boundary-parallel (Proposition~\ref{prop:M_g-atoroidal}).
\end{itemize}
Assuming these for now, it remains to be shown that $M_g$ has no essential disks or annuli.  The argument below is entirely formal, using only the above properties of $M_g$ and the fact that it is orientable with boundary a single torus.

We first observe that any torus $T \subset M_g$ which admits a compressing disk $D$ must bound a solid torus $S^1\times D^2 \subset M_g$.  Indeed, a neighborhood of $T\cup D$ has in its boundary a 2-sphere, which by irreducibility bounds a ball $B$.  If $T$ lies inside $B$, then it bounds an $S^1\times D^2$ inside $B$.  Otherwise a neighborhood of $B \cup D$ is a solid torus with boundary $T$.

Now if $M_g$ contains an essential disk $D$, then the above argument says that $\partial M_g$ bounds a solid torus, hence $M_g \cong S^1 \times D^2$.  But then $M_g$ would be Seifert fibered, so no such disk can exist.

Next, we suppose that $M_g$ contains an essential annulus $A$; the pair of curves $\partial A$ must be parallel on the torus $\partial M_g$, which they divide into a pair of annuli $A_1$ and $A_2$.  If the two components of $\partial A$ have the same orientation on $\partial M_g$, then each $A \cup A_i$ is a Klein bottle $K_i$.  Pushing $K_1$ slightly into the interior of $M_g$, it has a tubular neighborhood $N_1$ which (by the orientability of $M_g$) is the twisted $I$-bundle over $K_1$, and whose boundary is a torus $T_1$.  Since $M_g$ is not Seifert fibered it cannot be homeomorphic to $N_1$, so $T_1$ is not boundary-parallel, hence it is compressible.  We conclude as above that $T_1$ bounds a solid torus, but this is impossible since it bounds $N_1 \not\cong S^1\times D^2$ on one side and contains $\partial M_g$ on the other.

We have shown that the two components of $\partial A$ must be oppositely oriented on $\partial M_g$, so we form two tori $T_i = A \cup A_i$ for $i=1,2$.  The $T_i$ cannot be boundary-parallel, since then $A$ can be isotoped into $\partial M_g$, so they are compressible and hence bound solid tori.  We now see that $A$ separates $M_g$, since otherwise the $T_i$ would be nonseparating, and that $M_g \setminus N(A)$ is a pair of solid tori bounded by $T_1$ and $T_2$.  The core $c$ of $A$ cannot bound a disk in either solid torus since $A$ is essential, so each solid torus admits a Seifert fibration in which $c$ is a fiber, and these glue together to give a Seifert fibration of $M_g$, which does not exist.
\end{proof}

We now prove the claims about $M_g$ which were assumed in the proof of Theorem~\ref{thm:M_g-hyperbolic}.

\begin{lemma} \label{lem:mg-irreducible}
Each $M_g$ is irreducible and is not Seifert fibered.
\end{lemma}

\begin{proof}
For the irreducibility, we observe that any closed, oriented, embedded surface $S \subset M_g$ must be separating: if it is nonseparating, then it remains nonseparating in any Dehn filling of $M_g$ and hence determines a non-torsion element of $H_2(M_g(r_2)) = H_2(L(2g+5,2)) = 0$, contradiction.  So if $M_g$ contains an embedded sphere $S$ which does not bound a ball, then we can decompose along $S$ to write $M_g = Y \# N$, where $Y \not\cong S^3$ is closed and $\partial N = T^2$, and then $Y$ is a connected summand of both $M_g(r_{11}) = L(11g+3,11)$ and $M_g(r_{13}) = L(13g+8,13)$, again a contradiction.

Suppose now that $M_g$ is Seifert fibered, and let $r$ be the fiber slope.  The Seifert fibering extends across any Dehn filling of $M_g$ except possibly $M_g(r)$, so if the base is non-orientable then at least two of the three lens space fillings admit Seifert fibrations with non-orientable base.  But then this base must be $\RP^2$ by Proposition~\ref{prop:possible-bases}, and Corollary~\ref{cor:lens-space-rp2} says that any such lens space has order a multiple of 4; this is never the case for $L(2g+5,2)$, and it cannot hold for both $L(11g+3,11)$ and $L(13g+8,13)$ since the difference $2g+5$ in their orders is odd.  Thus the base orbifold must be orientable.  In this case, Heil \cite{heil} showed that $M_g(r)$ must be a connected sum of lens spaces and copies of $S^1 \times S^2$.  But $M_g(r_T) = Y_g$ is neither Seifert fibered nor a connected sum of such manifolds, so $M_g$ must not be Seifert fibered after all.
\end{proof}

The following lemma will be useful in proving that $M_g$ is atoroidal.

\begin{lemma} \label{lem:11-13-torus-surgery}
If $L(11g+3,11)$ and $L(13g+8,13)$ both arise as surgeries on the same torus knot for some fixed $g \geq 1$, then that torus knot is $T_{\pm 3,17}$.
\end{lemma}

\begin{proof}
Let $T_{a,b}$ be the torus knot in question.  Its lens space surgeries have slopes $ab + \frac{1}{n}$ for nonzero integers $n$, hence first homology of order $|nab+1| \equiv \pm1 \pmod{|ab|}$.  In particular, $ab$ must divide both $11g+3+\delta$ and $13g+8+\epsilon$ for some $\delta,\epsilon \in \{\pm1\}$, and hence it divides
\[ 11(13g+8+\epsilon) - 13(11g+3+\delta) = 49 + 11\epsilon - 13\delta \in \{25, 47, 51, 73 \}. \]
But $a$ and $b$ are coprime and different from $\pm1$, so their product can only divide $51 = 3\cdot 17$ since the rest of these are prime powers, and then $|a|$ and $|b|$ are equal to $3$ and $17$ in some order.
\end{proof}

Suppose now that $M_g$ contains an incompressible torus $S$ which is not boundary-parallel.  Again, $S$ must be separating, so we write
\[ M_g = N \cup_S P \]
where $\partial N = S$ and $\partial P = S \sqcup \partial M_g$.  We note that $P$ must also be irreducible and $\partial$-irreducible.  For the irreducibility, any essential sphere must again be separating, so it gives us a decomposition $P = P_1 \# P_2$ where $S \subset \partial P_1$, and we have $M_g = (N \cup_S P_1) \# P_2$.  Either $P_2$ is closed and not $S^3$, or its boundary is $\partial M_g$ and then $N \cup_S P_1$ is closed and not $S^3$ since it has an incompressible torus.  Either way, every filling of $M_g$ has a connected summand in common, but this is clearly false for $M_g(r_{11})$ and $M_g(r_{13})$, so $P$ is irreducible is claimed.  For the $\partial$-irreducibility, we argue as in Theorem~\ref{thm:M_g-hyperbolic} that since $P$ is irreducible, any compressing disk would force the corresponding component of $\partial P$ to bound a solid torus in $P$, which is absurd.

\begin{proposition} \label{prop:not-cabled}
The manifold $P$ is not a cable space.
\end{proposition}

\begin{proof}
Suppose that $P$ is the complement of a $(p,q)$-cable in $S^1\times D^2$, with $S = S^1 \times \partial D^2$ and $\partial M_g$ the boundary of a neighborhood of the cable.  Let $r_c = pq$ denote the slope of the cabling annulus.  Then the fillings $P(r)$ are classified by \cite[Lemma~7.2]{gordon}: we have
\[ P(r) = \begin{cases} S^1\times D^2 \# L(q,p) & r=r_c \\ S^1\times D^2 & \Delta(r,r_c)=1, \end{cases} \]
and $P(r)$ is a Seifert fibered space with incompressible boundary if $\Delta(r,r_c) > 1$.

This last case cannot happen for any $r \in \{r_2,r_{11},r_{13}\}$, because if the torus $S$ is incompressible in both $N$ and $P(r)$ then it remains so in the lens space $M_g(r)$, which is atoroidal.  Then one of them must be $r_c$, because otherwise the four slopes $r_c$, $r_2$, $r_{11}$, $r_{13}$ all have pairwise distance 1 and this is impossible.  The filling
\[ M_g(r_c) = N \cup_S P(r_c) = N \cup_S (S^1\times D^2 \# L(q,p)) \]
is a lens space, so the corresponding $N \cup_S (S^1\times D^2)$ must be $S^3$, and in particular $N$ is the complement of a knot $K$ in $S^3$.

The knot $K$ has two nontrivial lens space surgeries, since $P(r) = S^1\times D^2$ for whichever two of $r_2$, $r_{11}$, and $r_{13}$ differ from $r_c$.  If $K$ is not a torus knot, then the cyclic surgery theorem \cite{cgls} says that the slopes of these surgeries must be consecutive integers.  But the orders of these lens spaces are among
\[ 2g+5, \quad 11g+3, \quad 13g+8, \]
and no two of these are consecutive integers, so $K$ must be a torus knot.

We now consider the toroidal Dehn filling
\[ Y_g = M_g(r_T) = N \cup_S P(r_T). \]
If $\Delta(r_c,r_T) = 1$ then $P(r_T) = S^1\times D^2$ and so $Y_g$ is a Dehn filling of $N$, contradicting the fact that $Y_g$ is not surgery on a torus knot by \cite{moser}.  We must therefore have $\Delta(r_c,r_T) \geq 2$, so then $r_c = r_2$ and $P(r_T)$ is Seifert fibered with incompressible boundary $S$ which remains incompressible in $Y_g = N \cup_S P(r_T)$.  Then $S$ must be the unique incompressible torus in $Y_g$ up to isotopy, so $N$ is either $E_{2,3}$ or $E_{2,2g+1}$.  But $M_g(r_{11}) = L(11g+3,11)$ and $M_g(r_{13}) = L(13g+8,13)$ are both Dehn fillings of $N$, so Lemma~\ref{lem:11-13-torus-surgery} says that $N$ is actually $E_{\pm3,17}$ and we have a contradiction.
\end{proof}

\begin{proposition} \label{prop:M_g-atoroidal}
Every incompressible torus $S \subset M_g$ is boundary-parallel.
\end{proposition}

\begin{proof}
Let $S$ be an incompressible torus which is not boundary-parallel.  We first claim that $S$ remains incompressible inside $M_g(r_T) = Y_g$.  If not, then it also compresses in the atoroidal $M_g(r_2) = L(2g+5,2)$, and $\Delta(r_2,r_T)=2$, so \cite[Theorem~2.0.1]{cgls} asserts that $S$ and $\partial M_g$ cobound a cable space, contradicting Proposition~\ref{prop:not-cabled}.  Since $S$ is incompressible in $Y_g$, Proposition~\ref{prop:splicing-facts} tells us that $S$ is isotopic to the torus which separates $Y_g$ into two torus knot exteriors.  Thus $N$ and $P(r_T)$ are $E_{2,3}$ and $E_{2,2g+1}$ in some order, each with incompressible boundary $S$.

The torus $S \subset P$ must compress in each of the fillings $P(r_2)$, $P(r_{11})$, and $P(r_{13})$, because otherwise it would remain incompressible in the corresponding lens space fillings of $M_g$.  Then a neighborhood of $\partial P(r_2)$ together with a compressing disk is $(S^1\times D^2) \setminus B^3$, and likewise for $P(r_{11})$ and $P(r_{13})$, so we can write
\[ P(r_k) = (S^1\times D^2) \# Y_k, \qquad k=2,11,13 \]
for some closed manifolds $Y_k$.  In particular, the lens space $M_g(r_k)$ is the connected sum of $Y_k$ with some Dehn filling of the torus knot complement $N$, and so one of $Y_k$ and this Dehn filling are $S^3$ while the other is the lens space.

Lemma~\ref{lem:11-13-torus-surgery} tells us that $Y_{11}$ and $Y_{13}$ cannot both be $S^3$, so we fix $i$ and $j$ to be $11$ and $13$ in some order so that $Y_i$ is a lens space, hence $P(r_i) = (S^1\times D^2) \# Y_i$ is reducible.  We apply the main theorem of \cite{scharlemann}, with (in the notation of \cite{scharlemann}) $M = P(r_j)$ and $K\subset M$ the core of the $r_j$-filling and with $M' = P(r_i)$.   Then $M \setminus N(K) = P$ is irreducible and $\partial$-irreducible, and $\partial M$ compresses in $M = (S^1\times D^2) \# Y_j$ while $M'$ is neither $S^1\times D^2$ nor irreducible, so we conclude that $K$ is cabled.  If we replace $S$ with the boundary of the cable space then this also contradicts Proposition~\ref{prop:not-cabled}.  We conclude that $S$ cannot exist.
\end{proof}

This completes the proof of Theorem~\ref{thm:M_g-hyperbolic}. \hfill\qedsymbol

\begin{remark} \label{rem:infinitely-many-example}
We can identify a tangle $T$ in Figure~\ref{fig:knot-Lg} as the complement of the box with $g$ half-twists.  The branched double cover $Y=\dcover(T)$ is another 3-manifold with torus boundary, one which does not depend on $g$, and filling it in with the branched double cover of the $g$ half-twist tangle which was removed exhibits $Y(T_{2,3},T_{2,2g+1})$ as a Dehn filling of $Y$.  Thus $Y$ has infinitely many $SU(2)$-cyclic Dehn fillings.

We note that $Y$ is neither hyperbolic nor Seifert fibered.  Indeed, in the first case all but finitely many Dehn fillings would be hyperbolic, and in the second case all but at most one filling would be Seifert fibered.  However, none of the $Y(T_{2,3},T_{2,2g+1})$ are hyperbolic since they have incompressible tori, and they are not Seifert fibered either by Proposition~\ref{prop:splicing-facts}.
\end{remark}

\begin{example} \label{ex:four-sporadic}
By examining \cite{dunfield-exceptional} we have found that the hyperbolic manifolds $m259$ and $s337$ also have four $SU(2)$-cyclic Dehn fillings each, namely
\begin{align*}
m259(0,1) &= L(9,2), & m259(1,0) &= L(45,19), \\
m259(1,1) &= S^2((3,2),(3,2),(3,2)), & m259(2,1) &= \pm Y(T_{2,5},T_{2,5}); \\[1em]
s337(0,1) &= L(3,1), & s337(1,0) &= L(69,19), \\
s337(-1,1) &= S^2((3,1),(3,2),(3,5)), & s337(-2,1) &= \pm Y(T_{2,5},T_{2,-7}).
\end{align*}
It would not be surprising to see that these fall into an infinite family of examples, just as in Theorem~\ref{thm:four-fillings}.
\end{example}

\bibliographystyle{myalpha}
\bibliography{References}

\end{document}